\documentclass[a4paper,11pt]{article}
\usepackage{mathrsfs}
\usepackage{amsmath}                    
\usepackage{amssymb}                    
\usepackage{amsthm}                     
\usepackage{amsfonts}                   
\usepackage{color}
\usepackage{graphics,graphicx}
\usepackage{enumerate}
\usepackage{epsfig}
\usepackage{subcaption}
\usepackage{array}

\usepackage{multirow}
\usepackage[english]{algorithm}
\usepackage{dsfont}

\allowdisplaybreaks

\newtheorem{theorem}{Theorem}[section]

\newtheorem{proposition}[theorem]{Proposition}
\newtheorem{lemma}[theorem]{Lemma}
\newtheorem{remark}[theorem]{Remark}

\numberwithin{equation}{section}
\renewcommand{\epsilon}{\varepsilon}

\newcommand{\bigo}{\mathcal{O}}

\newcommand{\e}{\mathrm{e}}
\newcommand{\dd}{\mathrm{ds}}
\title{
Strang splitting method for semilinear parabolic problems with inhomogeneous boundary conditions: a correction based on the flow of the nonlinearity
}

\author{ 
Guillaume Bertoli\textsuperscript{1} and Gilles Vilmart\textsuperscript{1}
}

\begin{document}
\footnotetext[1]{
Universit\'e de Gen\`eve, Section de math\'ematiques, 2-4 rue du Li\`evre, CP 64, CH-1211 Gen\`eve 4, Switzerland, Guillaume.Bertoli@unige.ch, gilles.vilmart@unige.ch}

\maketitle

\begin{abstract}
The Strang splitting method, formally of order two, can suffer from order reduction when applied to semilinear parabolic problems with inhomogeneous boundary conditions. The recent work [L. Einkemmer and A. Ostermann. Overcoming order reduction in diffusion-reaction splitting.\thinspace Part 1.\thinspace Dirichlet boundary conditions. \textit{SIAM J. Sci. Comput.}, 37, 2015.\thinspace Part 2:\thinspace Oblique boundary conditions, \textit{SIAM J. Sci. Comput.}, 38, 2016] introduces a modification of the method to avoid the reduction of order based on the nonlinearity. In this paper we introduce a new correction constructed directly from the flow of the nonlinearity and which requires no evaluation of the source term or its derivatives. The goal is twofold. One, this new modification requires only one evaluation of the diffusion flow and one evaluation of the source term flow at each step of the algorithm and it reduces the computational effort to construct the correction. Second, numerical experiments suggest it is well suited in the case where the nonlinearity is stiff. We provide a convergence analysis of the method for a smooth nonlinearity and perform numerical experiments to illustrate the performances of the new approach.

\medskip


\medskip


\end{abstract}

\section{Introduction}
In this paper, we consider a parabolic differential equation of the form
\begin{align}\label{eq:D+f}
\partial_t u = D u + f(u)\quad \mathrm{in}\ \Omega, \qquad Bu = b\quad \mathrm{on}\ \partial\Omega, \qquad u(0)=u_0,
\end{align}
where $D$ is a linear diffusion operator and $f$ is a nonlinearity.
A natural method for approximating~\eqref{eq:D+f} are splitting methods. The idea is to divide the main equation~\eqref{eq:D+f} into two auxiliary subproblems~\eqref{eq : f} and~\eqref{eq : D} so one can use specific numerical methods for both subproblems to enhance the global efficiency of the computation of~\eqref{eq:D+f}. 
Let $N\in\mathbb{N}$ and let $\tau=\frac{T}{N}$ be the time step. Then, one step of the classical Strang splitting is either
\begin{equation}\label{eq : classical Strang 1}
u_{n+1}=\phi^D_{\frac{\tau}{2}}\circ\phi^f_{\tau}\circ\phi^D_{\frac{\tau}{2}}(u_n)
\end{equation}
or alternatively
\begin{equation}\label{eq : classical Strang 2}
u_{n+1}=\phi^f_{\frac{\tau}{2}}\circ\phi^D_{\tau}\circ\phi^f_{\frac{\tau}{2}}(u_n),
\end{equation}
where $\phi^f_{t}(u_0)$ is the flow after time $t$ of
 \begin{align}\label{eq : f}
\partial_t u = f(u), \qquad u(0)=u_0,
\end{align}
and  $\phi_t^{D}(u_0)$ is the flow after time $t$ of 
\begin{align}\label{eq : D}
\partial_t u = D u \quad \mathrm{in}\ \Omega, \qquad Bu = b\quad \mathrm{on}\ \partial\Omega, \qquad u(0)=u_0.
\end{align}
The Strang splitting, when applied to ODE with a sufficiently smooth solution, is a method of order two. However, when the Strang splitting is applied to solve the problem~\eqref{eq:D+f}, a reduction of order can be observed in the case of non homogeneous boundary conditions as shown in~\cite{Ost15,Ost16}. The reason is that $Bu$ is not left invariant through the flow $\phi_t^f$ and therefore leaves the domain of $D$, which creates a discontinuity at $t=0$ in the flow $\phi_t^D$. In this case the Strang splitting has in general a fractional order of convergence between one and two \cite[section 4.3]{Ost16}.
In~\cite{Ost15,Ost16}, a modification of the Strang splitting is given to recover the order two. The main idea in~\cite{Ost16} is to find a function $q_n$ such that $Bu$ is now left invariant by $\phi_t^{f-q_n}$,
the exact flow of 
\begin{align}\label{eq : f-q}
\partial_t u = f(u)-q_n.
\end{align}
One step of the modified splitting in~\cite{Ost16} is then
\begin{equation}\label{eq : modified Strang 1}
u_{n+1}=\phi^{D+q_n}_{\frac{\tau}{2}}\circ\phi^{f-q_n}_{\tau}\circ\phi^{D+q_n}_{\frac{\tau}{2}}(u_n),
\end{equation}
where $\phi_t^{D+q_n}$ is the exact flow of 
\begin{align}\label{eq : D+q}
\partial_t u = D u + q_n\quad \mathrm{in}\ \Omega, \qquad  \qquad Bu = b\quad \mathrm{on}\ \partial\Omega.
\end{align}
Numerically, one can choose any smooth function $q_n$ such that
 \begin{equation}\label{eq : correction 3 parts}
 Bq_n=Bf(u_n)+\bigo(\tau)\quad \mathrm{on}\ \partial\Omega.
 \end{equation}
 Several options to construct $q_n$ are presented in \cite{Ein18}. One challenge is then to find a correction $q_n$ that is both cheap to compute and minimizes the constant of error. 
 
In this paper, we give a new modification of the classical Strang splitting that removes the order reduction and requires only one evaluation of the diffusion flow and one evaluation of the source term flow at each step of the algorithm and allows a cheaper and easier to implement construction of $q_n$. As illustrated in the experiments, this new construction performs better for the case of a stiff reaction. The idea is to leave $Bu$ unpreserved at the boundary through the flow $\phi^{f}_{\tau}$ and then apply a correction $q_n$ afterwards that brings back the solution to the domain of $D$. This new splitting is then
$$
\mathcal{S}_{\tau}(u_n)=(\phi^{f}_{\frac{\tau}{2}}\circ\phi^{-q_n}_{\frac{\tau}{2}}\circ\phi^{D+q_n}_{\tau}\circ\phi^{-q_n}_{\frac{\tau}{2}}\circ\phi^{f}_{\frac{\tau}{2}})(u_n)
$$ 
and the correction $q_n$ is constructed such that $Bq_n=\frac{2}{\tau}(B\phi^f_{\tau/2}(u_n)-Bu_n)$ on $\partial\Omega$. The correction $q_n$ is now constructed from the output of the flow $\phi^f_{\tau/2}(u_n)$ and not directly from the nonlinearity $f$.  Note that the computation of $\phi^{-q_n}_{\frac{\tau}{2}}$ and $q_n$ requires no evaluation of $f$, which is particularly useful when $f$ is costly. More importantly, in many situations, the flow $\phi^f_{\tau/2}(u_n)$ is smoother than the nonlinearity $f$ itself, which can avoid the possible instability due to the eventual stiffness of the reaction.

In section 2, we give the appropriate framework for the convergence analysis of this modified splitting. In section 3, we describe the new modification we consider in this paper. In section 4, we prove that the method is of global order two under the hypotheses made in section 2. In section 5, we present some numerical experiments to illustrate the performance of the new approach.
\section{Analytical framework}\label{section : framework}
In this section, we describe the appropriate analytical framework that we consider in this paper. We choose the framework described in \cite[Chapter 3]{Lun95}. The notation is similar to that used in \cite{Lun95}.
Let $\Omega\subset\mathbb{R}^d$ be a bounded connected open set with a $C^2$ boundary $\partial \Omega$. We consider the following semilinear parabolic problem on $\Omega\times [0,T]$, $T\geq0$. 
\begin{align*}
\partial_t u = D u + f(u)\quad \mathrm{in}\ \Omega, \qquad Bu = b\quad \mathrm{on}\ \partial\Omega, \qquad u(0)=u_0.
\end{align*}
The differential operator $D$ is defined by 
$$
D=\sum_{i,j=1}^n a_{ij}(x)\partial_{ij}+\sum_{i=1}^n a_{i}(x)\partial_i+a(x)I,
$$ 
where the matrix $(a_{ij}(x))\in\mathbb{R}^{d\times d}$ is assumed symmetric and there exists $\lambda>0$ such that
$$
\forall x\in  \overline{\Omega},\quad \forall \xi \in \mathbb{R}^n,\quad \sum_{i,j=1}^n a_{ij}(x) \xi_i \xi_j \geq \lambda | \xi |^2,
$$
and $a_{ij},\ a_i,\ a$ are assumed continuous, $a_{i,j},\ a_i,\ a\in C(\Omega, \mathbb{R}) $ .
Let $B$ be the linear operator
$$
B=\sum_{i=1}^n\beta_i(x)\partial_i+\alpha(x)I,
$$
where we assume the uniform non tangentiality condition
$$
\inf_{x\in\partial\Omega}\left|\sum_{i=1}^n\beta_i(x)\nu_i(x)\right|>0,
$$
where $\nu(x)$ is the exterior normal unit vector at $x\in\partial\Omega$. We assume that the functions $\beta_i$ and $\alpha$ are continuously differentiable, $\beta_i,\alpha \in C^1(\partial \Omega, \mathbb{R})$ and $b$ is continuously differentiable, $b\in C^1([0,T],W^{2,p}(\partial \Omega))$.
We follow next the construction made in \cite{Ost16} to take benefit of homogeneous boundary conditions.
Let $z\in C^1([0,T], W^{2,p}(\Omega))$ satisfying $Bz=b$ on $\partial\Omega$.
We define $\tilde{u}=u-z$ and~\eqref{eq:D+f} becomes
 \begin{align}\label{eq:A+f}
\partial_t \tilde{u} = D \tilde{u} + f(\tilde{u}+z)+Dz-\partial_t z\quad \mathrm{in}\ \Omega, \qquad B\tilde{u} = 0\quad \mathrm{on}\ \partial\Omega, \qquad \tilde{u}(0)=u_0-z(0).
\end{align}
We define the linear operator $A$ as 
$$Av=Dv \quad \forall v\in D(A)=\{u\in W^{2,p}(\Omega)\ :\ Bu=0\ \mathrm{in}\ \partial\Omega \}.$$
Under those conditions $-A$ is a sectorial operator and therefore $A$ is the generator of an analytic semigroups  $e^{tA}$ (see \cite{Lun95}, Chapter 3).
In particular, the operator $A$ satisfies the parabolic smoothing property that we use intensively throughout this paper,
\begin{align}
\|(-A)^{\alpha}\e^{t A}\|\leq \frac{C}{t^\alpha}\quad\alpha\geq 0\quad t>0.
\end{align}
We denote 
\begin{align}\label{eq : phi1}
\varphi_1(\tau A)=\int_0^\tau \e^{(\tau-s)A}\frac{1}{\tau}\dd.
\end{align}
We observe that $\tau A\varphi_1(\tau A)=\bigo(1)$ is a bounded operator.
We recall the following theorem, a direct consequence of \cite[Theorem 8.1$^\prime$]{Gri67}, valid for $1<p<\infty$ (see also~\cite[Theorem 1 and 2]{Fuj67} in English for the case $p=2$), which states that there exists $\alpha>0$ such that $D((-A)^{\alpha})$ becomes free of the boundary conditions. 
\begin{theorem}\label{thm : fractional}
Let the differential operator $A$ be define as in section~\ref{section : framework}. Then, there exists $\alpha>0$ such that $$W^{1,p}(\Omega)\subset D((-A)^{\alpha}).$$
\end{theorem}
We ask $f$ to satisfy the following. Let $U\subset W^{2,p}(\Omega)$ be a neighborhood of the exact solution $u$. Then we require the nonlinearity $f$ to be twice continuously differentiable in $U$ with values in $W^{2,p}(\Omega)$, $f\in C^2(U,W^{2,p}(\Omega))$. We refer to the discussion in \cite[Section 4]{Ost16} for possibly relaxing the hypotheses made on $f$.
We assume that the solution $u$ of~\eqref{eq:D+f} is twice continuously differentiable, $u\in C^2([0,T], W^{2,p}(\Omega))$.
The exact solution of~\eqref{eq:D+f} can be expressed using the variation of constant formula,
\begin{align}\label{eq:exact solution}
u(t_{n+1})&=z_n(\tau)+\e^{\tau A}(u(t_n)-z_n(0))\\
&+\int_0^\tau \e^{(\tau-s)A}(f(u(t_n+s))+Dz_n(s)-\partial_t z_n(s))\dd.
\end{align}

\section{Description of the method}
In this paper, we describe a new modification for the Strang splitting that we call the five parts modified Strang splitting. 
The idea of this new modification is to compose the flow of the nonlinearity $w_n=\phi^f_{\frac{\tau}{2}}(u_n)$ with a projection $\phi^{-q_n}_{\frac{\tau}{2}}(w_n)=w_n-\frac{\tau}{2}q_n$, the exact flow of
\begin{equation}\label{eq : projection}
\partial_t u = -q_n\qquad u(0)=w_n,
\end{equation}
where $q_n$ is independent of time, in the spirit of projection methods used in the context of geometric numerical integration (see~\cite[Chapter IV.4]{hlw10}).
The splitting algorithm that we propose and analyze in this paper is given by
\begin{equation}\label{eq : five parts Strang splitting}
u_{n+1}=\mathcal{S}_n(u_n) = \phi^f_{\frac{\tau}{2}}\circ\phi^{-q_n}_{\frac{\tau}{2}}\circ\phi^{D+q_n}_{\tau}\circ\phi^{-q_n}_{\frac{\tau}{2}}\circ\phi^f_{\frac{\tau}{2}}(u_n),
\end{equation}
where $\phi^{D+q_n}_{\tau}$ is the flow of \eqref{eq : D+q} and
$(q_n)_{n\in\{0,\ldots,N\}}$ is a sequence of correctors satisfying one of the two following conditions on the boundary $\partial\Omega$,
\begin{align}\label{eq : q condition}
Bq_n&=\frac{2}{\tau}(B\phi^f_{\tau/2}(u_n)-B(u_n)),
\end{align}
or alternatively 
\begin{align}\label{eq : q condition 2}
Bq_n&=\frac{2}{\tau}(B\phi^f_{\tau/2}(u_n)-b_n)
\end{align}
(see Remark~\ref{rem : composition} below). In the interior of $\Omega$, we require $q_n$ to be in $ W^{2,p}(\Omega)$. A possibility, to construct $q_n$ in $\Omega$, is to choose $q_n$ to be harmonic if this is possible or to use a smoothing iterative algorithm. For more details on how to construct the correction $q_n$ on the interior of the domain, see~\cite{Ein18}. We also assume $(q_n)_{n\in\{0,\ldots,N\}}$ uniformly bounded, that is there exists a constant $C$ independent of $n,\tau$ and $N$, such that $\|q_n\|_{L^p(\Omega)}\leq C$.
We observe that $\frac{2}{\tau}(B\phi^f_{\tau/2}(u_n)-Bu_n)$ is a finite difference approximation of $\partial_t \phi^f_{\tau/2}(u_n)=f(u_n)$, and hence this new condition is close to~\eqref{eq : correction 3 parts}. 
\begin{remark}
In contrast to the correction of~\cite{Ost16}, the correction $q_n$ for the five parts modified Strang splitting~\eqref{eq : five parts Strang splitting} is constructed directly from the flow of the nonlinearity $\phi^f_{t}$ and not from the nonlinearity $f$ itself. The modified splitting of~\cite{Ost16} has a good behavior when the nonlinearity $f$ is not stiff and cheap to compute as analyzed and illustrated numerically in~\cite{Ost16}. However, in the case of a stiff nonlinearity $f$, the modification for the splitting in~\cite{Ost16} can lead to instability in contrast to~\eqref{eq : five parts Strang splitting} as shown in the experiments (see section \ref{section : Numerical experiments }). Furthermore, if the nonlinearity $f$ is very costly to compute, the correction in~\cite{Ost16} requires an additional cost that can be substantial. In comparison, the construction of the correction $q_n$ for~\eqref{eq : five parts Strang splitting} requires no evaluation of $f$ or its derivatives. We also observe that, in the extreme case where the diffusion D is zero, the five parts modified Strang splitting~\eqref{eq : five parts Strang splitting} becomes exact analogously to the classical Strang splitting methods~\eqref{eq : classical Strang 1} and~\eqref{eq : classical Strang 2}. This later property does not hold for the modified splitting in~\cite{Ost16}. (Note that the flows $\phi^q_{\frac{\tau}{2}}$ and $\phi^{f-q}_{\tau}$ do not commute in general). 
\end{remark}
\begin{remark}\label{rem : composition}
When implementing the classical Strang splitting,  it is often computationally advantageous to compose the flows $\phi^f_{\frac{\tau}{2}}$ that appear in the splitting, that is $\phi^f_{\frac{\tau}{2}}\circ \phi^f_{\frac{\tau}{2}}=\phi^f_{\tau}$. The numerical approximation $u_n$ of $u(t_n)$ is then
$$
u_n=\phi^f_{\frac{\tau}{2}}\circ\phi^D_{\tau}\circ(\phi^f_{\tau}\circ\phi^D_{\tau})^{n-1}\circ\phi^f_{\frac{\tau}{2}}(u_0),
$$
which makes the classical Strang splitting have the same cost as the Lie Trotter splitting with only one evaluation of $\phi^f_{\tau}$ per time step.
If we use the correction~\eqref{eq : q condition}, we need then to compute $Bu_k$ and this idea does not apply since the algorithm requires $Bu_k$. However, if we use the correction~\eqref{eq : q condition 2} instead, we can implement the five parts modified Strang splitting~\eqref{eq : five parts Strang splitting} as explained above for the classical Strang splitting. Note that this is an advantageous implementation that cannot be used with the method presented in \cite{Ost16}. 
\end{remark}
\section{Convergence analysis}\label{section : analysis}
We prove in this section that, using the framework and assumptions described in section~\ref{section : framework}, the five parts modified Strang splitting method~\eqref{eq : five parts Strang splitting} is of global order of convergence two and thus avoids order reduction phenomena.
\begin{theorem}\label{thm : global error fractional}
Under the assumption of section~\ref{section : framework} the five parts modified Strang splitting~\eqref{eq : five parts Strang splitting} satisfies the bound 
$$\|u_n-u(t_n)\|\leq C\tau^2 \quad 0\leq n\tau \leq T,$$
for all $\tau$ small enough, and where the constant C depends on T but is independent on $\tau$ and n.
\end{theorem}
We start by showing the following proposition, which states that the five parts splitting is at least first order convergent.
\begin{proposition}\label{thm : global order one}
Under the assumptions of section~\ref{section : framework} the five parts modified Strang splitting~\eqref{eq : five parts Strang splitting} satisfies
$$\|u_n-u(t_n)\|\leq C\tau \quad 0\leq n\tau \leq T,$$
for all $\tau$ small enough, and where the constant C depends on T but is independent on $\tau$ and n.
\end{proposition}
 We need this result to justify the condition~\eqref{eq : q condition} and~\eqref{eq : q condition 2}  for the construction of $q_n$, that is we need to show that $q_n$ satisfies
\begin{equation*}
f(u(t_n+s))-q_n=\phi_n+\bigo(\tau),
\end{equation*}
with $\phi_n\in D(A)$ and $A\phi_n=\bigo(1)$. 

The proof of Theorem~\ref{thm : global error fractional} relies on Theorem~\ref{thm : fractional} (from \cite{Gri67}). The proof of Theorem~\ref{thm : fractional} uses sophisticated tools from interpolation theory. Since all the arguments in our proofs do not require any knowledges of interpolation theory, we decide to present first the proof without using Theorem~\ref{thm : fractional} and obtain Proposition~\ref{thm : global error} below. We then explain how we use Theorem~\ref{thm : fractional}.
\begin{proposition}\label{thm : global error}
Under the assumptions of section~\ref{section : framework} the five parts modified Strang splitting~\eqref{eq : five parts Strang splitting} satisfies 
$$\|u_n-u(t_n)\|\leq C\tau^2(1+|\log\tau|) \quad 0\leq n\tau \leq T,$$
for all $\tau$ small enough, and where the constant C depends on T but is independent on $\tau$ and n.
\end{proposition}
For all the convergence analysis, for a function $\phi$ in $L^p(\Omega)$, we shall use the following notation for $k=0,1,\ldots$: 
\begin{equation}\label{eq : O Definition}
 \phi=\bigo(\tau^k)\ \ \text{if}\ \ \|\phi\|_{L^p(\Omega)}\leq C\tau^k,
\end{equation}
  where $C$ is independent of $\tau$ and where $\tau$ is assumed small enough.
\subsection{Quadrature error analysis}
The main idea of the convergence analysis is to approximate the integrals of the form $\int_0^\tau \e^{(\tau - s)A}\widehat{\psi}(s)ds$ with quadrature formulas, using the Peano integral representation of the error. This idea is not new in the literature and is used, for example, in \cite{Jah00,Ost16,Ost15}.
Indeed one cannot apply quadrature formulas naively to such integrals since, for a smooth function $s\mapsto\hat{\psi}(s)$, the integrand $\e^{(\tau - s)A}\widehat{\psi}(s)$ is not smooth enough in general with respect to $s$. We recall that to state that the (local) error of a $k$th order quadrature formula is $\bigo(\tau^{k+1})$, the integrand needs to be $k$ times continuously differentiable with respect to $s$ but this assumption does not hold in the proofs in general. Hence, we shall prove refined estimates for the order $k=2,3$ (see Lemmas~\ref{lemma : Quadrature Error lemma}, \ref{lemma : Quadrature Error lemma 2}, and~\ref{lemma : Quadrature Error lemma fractional}). We show  in Lemma~\ref{lemma : Quadrature Error lemma} below, with the help of the parabolic smoothing property, that if $\widehat{\psi}$ is close to the domain of $A$, that is if
\begin{align}\label{eq:QuadratureCondition}
\widehat{\psi}(s)=\phi_0+\bigo(\tau), \quad \phi_0\in D(A),\quad A\phi_0=\bigo(1),
\end{align}
is satisfied, then first and second order quadrature formulas regain partially their accuracy. 
In \cite{Ost16}, the authors prove this statement is true for the left rectangle quadrature formula and the midpoint rule. Since we need such results for various quadrature formulas, we prove instead it is true for a general quadrature formula since it adds no difficulties to the proof. The first of the two lemmas that follow deals with quadrature formulas of order one. The second lemma deals with quadrature formulas of order two.
\begin{lemma}\label{lemma : Quadrature Error lemma}
Let $s\mapsto \widehat{\psi}(s)$ be a continuously differentiable function with values in $L^p(\Omega)$, $1<p<\infty$, and let $\psi(s)=\e^{(\tau - s)A}\widehat{\psi}(s)$. Let $Q(\psi)=\tau\sum_{k=1}^{m} b_k \psi(\tau c_k)$ be a quadrature formula that approaches the integral $\int_0^\tau \psi(s)ds$ such that $\sum_{k=1}^{m} b_k =1$. 

Then the quadrature error $E$ satisfies 
\begin{equation*}
E:=\int_0^\tau \psi(s)ds -\tau\sum_{k=1}^{m} b_k \psi(\tau c_k)=\bigo(\tau).
\end{equation*}
If additionally $\widehat{\psi}(s)=\phi_0+\bigo(\tau)$ with $\phi_0\in D(A)$  and $A\phi_0=\bigo(1)$, then
$$E =\bigo(\tau^2).$$ 
\end{lemma}

\begin{proof}
Since $\psi$ is uniformly bounded on $[0,\tau]$, the first result follows.
Let us assume that the condition~\eqref{eq:QuadratureCondition} is satisfied.

Let us compute the first derivative of $\psi$,
\begin{align*}
\psi'(s)&=-A\e^{(\tau - s)A}\widehat{\psi}(s)+\e^{(\tau - s)A}\widehat{\psi}'(s).
\end{align*}
Let $Q_l(\psi)=\tau\psi(0)$ be the left rectangle quadrature formula. We prove that every first order quadrature formula $Q$ satisfies
\begin{equation}\label{eq:QuadratureEquivalence}
Q_l(\psi)- Q(\psi)=\bigo(\tau^2).
\end{equation}
First, we observe that $\left\|\tau A\varphi_1(\tau A)\widehat{\psi}(0)\right\|=\bigo(\tau)$. Indeed
$$
\left\|\tau A\varphi_1(\tau A)\widehat{\psi}(0)\right\|\leq\tau\left\|\varphi_1(\tau A)\right\|\left\|A\phi_0\right\|+\left\|\tau A\varphi_1(\tau A)\right\|\left\|\bigo(\tau)\right\|\leq C\tau.
$$
We extend $\psi(c\tau)$ around $0$,
\begin{align*}
\psi(0)&=\e^{\tau A}\widehat{\psi}(0)=\widehat{\psi}(0)+\tau A\varphi_1(\tau A)\widehat{\psi}(0)=\widehat{\psi}(c\tau)+\bigo(\tau)=\psi(c\tau)+\bigo(\tau),
\end{align*}
which proves~\eqref{eq:QuadratureEquivalence} since
\begin{align*}
\tau\psi(0)-\tau\sum_{k=0}^m b_k\psi(\tau c_k)=\tau\psi(0)-\tau\psi(0)\sum_{k=1}^m b_k+\bigo(\tau^2)=\bigo(\tau^2).
\end{align*}
Therefore, we only need to show that 
\begin{equation}\label{eq : first order quadrature error}
\int_0^\tau \psi(s)ds - Q_l(\psi) =\bigo(\tau^2).
\end{equation}
We write the quadrature error as follows using the Peano kernel representation of the error for a first order quadrature formula,
\begin{align*}
\int_0^\tau \psi(s)ds - Q(\psi) = \tau^2\int_0^1(1-s)\psi'(\tau s)ds-\tau^2\sum_{i=1}^m b_i\int_0^{c_i}\psi'(\tau s)ds,
\end{align*}
which gives for $Q_l$,
$$
\int_0^\tau \psi(s)ds - Q_l(\psi) = \tau^2\int_0^1(1-s)\psi'(\tau s)ds.
$$
We need to bound the integral $\int_0^1(1-s)\psi'(\tau s)ds$.
We first bound $\psi'(s)$. For that, we need to bound $\|-A\e^{\tau(1 - s)A}\widehat{\psi}(\tau s)\|$. We get
\begin{align*}
\left\|-A\e^{\tau(1 - s)A}\widehat{\psi}(\tau s)\right\|\leq\left\|\e^{\tau(1 - s)A}\right\|\left\|-A\phi_0\right\|+\left\|-A\e^{\tau(1 - s)A}\right\|\left\|\bigo(\tau) \right\|\leq C+\frac{\tau C}{\tau (1-s)}.
\end{align*}
Therefore
$$
\left\|\psi'(\tau s)\right\|=\left\|-A\e^{\tau(1 - s)A}\widehat{\psi}(\tau s)+\e^{\tau (1- s)A}\widehat{\psi}'(\tau s) \right\| \leq C(1+\frac{1}{1-s}).
$$
We can now compute the error of the quadrature formula $Q_l$ and show~\eqref{eq : first order quadrature error}.
This follows from the inequality
\begin{align*}
\left\|\tau^2\int_0^1(1-s)\psi'(\tau s)ds\right\|\leq \tau^2\int_0^1(1-s)C(1+\frac{1}{1-s})ds=C\tau^2.
\end{align*}
Which gives us, with~\eqref{eq:QuadratureEquivalence}, the desired result for any first order quadrature formula,
$$
\int_0^\tau \psi(s)ds - Q(\psi) =\bigo(\tau^2),
$$
which concludes the proof of Lemma~\ref{lemma : Quadrature Error lemma}.
\end{proof}

\begin{lemma}\label{lemma : Quadrature Error lemma 2}
Let $Q$ and $\psi$ be as in Lemma \ref{lemma : Quadrature Error lemma}. We assume that $\widehat{\psi}$ is twice continuously differentiable and that $Q$ is a second order quadrature formula ($\sum_{k=1}^{m} c_kb_k =\frac{1}{2}$).
Then
\begin{equation}\label{eq : Quadature error estimate 1}
E =A\bigo(\tau^2)+\bigo(\tau^2).
\end{equation}
If, in addition, $\widehat{\psi}(s)=\phi_0+\bigo(\tau)$ with $\phi_0\in D(A)$ and $A\phi_0=\bigo(1)$, then   
\begin{equation}\label{eq : Quadature error estimate 2}
E =A\bigo(\tau^3)+\bigo(\tau^3).
\end{equation}
\end{lemma}

Using the notation~\eqref{eq : O Definition}, we remark that the term $A\bigo(\tau^k)$, where $k=2$ in~\eqref{eq : Quadature error estimate 1} and $k=3$ in~\eqref{eq : Quadature error estimate 2}, is a function of the form $A\phi$, where $\|\phi\|_{L^p(\Omega)}\leq C\tau^k$. In particular $\phi\in L^p(\Omega)$ is not in the domain of $A$ in general, and hence $A\phi\notin L^{p}(\Omega)$ in general and $A\phi\in W^{-2,p}(\Omega)$.
\begin{proof}
The second derivative of $\psi$ is 
\begin{align*}
\psi''(s)&=A^2\e^{(\tau - s)A}\widehat{\psi}(s)-2A\e^{(\tau - s)A}\widehat{\psi}'(s)+\e^{(\tau - s)A}\widehat{\psi}''(s).
\end{align*}
If $Q$ is a second order quadrature formula we write the quadrature error as follows using the Peano kernel representation of the error for a second order quadrature formula:
\begin{align*}
\int_0^\tau \psi(s)ds - Q(\psi) = \tau^3\int_0^1\frac{(1-s)^2}{2}\psi''(\tau s)ds-\tau^3\sum_{i=1}^m b_i\int_0^{c_i}(c_i-s)\psi''(\tau s)ds.
\end{align*}
It remains to estimate $$P_1(\tau s)=A\e^{\tau(1 - s)A}\widehat{\psi}(\tau s)-2\e^{\tau(1 - s)A}\widehat{\psi}'(\tau s)$$ and $$P_2(\tau s)=\e^{\tau(1 - s)A}\widehat{\psi}''(\tau s).$$
We first bound $\|A\e^{\tau( 1- s)A}\widehat{\psi}(\tau s)\|$. We get
\begin{align*}
\left\|A\e^{\tau( 1- s)A}\widehat{\psi}(\tau s)\right\|\leq\left\|A\e^{\tau(1 - s)A}\right\|\left\|\widehat{\psi}(\tau s)\right\|\leq C\frac{1}{\tau(1-s)}.
\end{align*}
Then
\begin{align}\label{eq : P1}
\left\|P_1(\tau s)\right\| = \left\|A\e^{\tau(1 - s)A}\widehat{\psi}(\tau s)-2\e^{\tau(1 - s)A}\widehat{\psi}'(\tau s)\right\|
\leq C(1+\frac{1}{\tau(1-s)})
\end{align}
and 
\begin{align*}
\left\|P_2(\tau s)\right\|=\left\|\e^{\tau(1 - s)A}\widehat{\psi}''(s)\right\|\leq \left\|\e^{\tau( 1- s)A}\right\|\left\|\widehat{\psi}''(s)\right\|\leq C.
\end{align*}
This gives the following estimation for the integrals
$\int_0^1\frac{(1-s)^2}{2}P_i(\tau s)ds$:
\begin{align*}
\left\|\int_0^1\frac{(1-s)^2}{2}P_1(\tau s)ds\right\|&\leq C\int_0^1\frac{(1-s)^2}{2}(1+\frac{1}{\tau(1-s)})ds\leq \frac{C}{\tau},
\end{align*}
\begin{align*}
\left\|\int_0^1\frac{(1-s)^2}{2}P_2(\tau s)ds\right\|&\leq\int_0^1\frac{(1-s)^2}{2}Cds\leq C.
\end{align*}
We show that $\|\sum_{i=1}^m b_i\int_0^{c_i}(c_i-s)P_1(\tau s)ds\|\leq \frac{C}{\tau}$.
\begin{align*}
\left\|\sum_{i=1}^m b_i\int_0^{c_i}(c_i-s)P_1(\tau s)ds\right\|&\leq \sum_{i=1}^m b_i\int_0^{c_i}(c_i-s)\|P_1(\tau s)\|ds\\
&\leq C\sum_{i=1}^m b_ic_i + C\sum_{i=1}^m b_i \int_0^{c_i}(c_i-s)\frac{1}{\tau(1-s)}ds.\\ 
\end{align*}
If $c_i=1$, we have
$$
\int_0^1(1-s)\frac{1}{\tau(1-s)}ds=\int_0^1\frac{1}{\tau} ds=\frac{C}{\tau}.
$$
If $c_i\neq 1$, then
$$
\int_0^{c_i}(c_i-s)\frac{1}{\tau(1-s)}ds=\frac{C}{\tau}.
$$
Therefore 
$$
\left\|\sum_{i=1}^m b_i\int_0^{c_i}(c_i-s)P_1(\tau s)ds\right\|\leq \frac{C}{\tau}.
$$
For the integral of $P_2$, we obtain
\begin{align*}
\left\|\sum_{i=1}^m b_i\int_0^{c_i}(c_i-s)P_2(\tau s)ds\right\|&\leq \sum_{i=1}^m b_i\int_0^{c_i}(c_i-s)\|P_2(\tau s)\|ds\leq C,
\end{align*}
which gives the desired bound for the error,
\begin{align*}
&\int_0^\tau \psi(s)ds - Q(\psi)=A\bigo(\tau^2)+\bigo(\tau^2).
\end{align*}

If condition~\eqref{eq:QuadratureCondition} is satisfied we can obtain a better bound for $\|A\e^{\tau( 1- s)A}\widehat{\psi}(\tau s)\|$ and thus also for $\|P_1(\tau s)\|$,
\begin{align*}
\|A\e^{\tau( 1- s)A}\widehat{\psi}(\tau s)\|
&\leq\|\e^{\tau(1 - s)A}\|\|A\phi_0\|+\|-A\e^{\tau(1 - s)A}\|\|\bigo(\tau)\|\leq C+C\tau\frac{1}{\tau(1-s)}.
\end{align*}
We obtain the following estimation for $P_1(\tau s)$,
\begin{align}\label{eq : P1 modified}
\|P_1(\tau s)\| = \|A\e^{\tau(1 - s)A}\widehat{\psi}(\tau s)-2\e^{\tau(1 - s)A}\widehat{\psi}'(\tau s)\|
\leq C(1+\frac{1}{1-s}),
\end{align}
which gives us the estimation $\|\sum_{i=1}^m b_i\int_0^{c_i}(c_i-s)P_1(\tau s)ds\|\leq C$.
Finally, we have the desired error bound
\begin{align*}
&\int_0^\tau \psi(s)ds - \tau Q_1(\psi(\tau s))=A\bigo(\tau^3)+\bigo(\tau^3),
\end{align*}
which concludes the proof of Lemma~\ref{lemma : Quadrature Error lemma 2}.
\end{proof}

Using Theorem~\ref{thm : fractional}, we can improve Lemma~\ref{lemma : Quadrature Error lemma 2} as follows.
\begin{lemma}\label{lemma : Quadrature Error lemma fractional}
Under the hypotheses of Lemma~\ref{lemma : Quadrature Error lemma 2}, there exists $\alpha>0$ such that
$$
\int_0^\tau \psi(s)ds - \tau Q(\psi(\tau s)) =(-A)^{1-\alpha}\bigo(\tau^2)+\bigo(\tau^2).
$$
If additionally condition~\eqref{eq:QuadratureCondition} is satisfied, then
$$
\int_0^\tau \psi(s)ds - \tau Q(\psi(\tau s)) =(-A)^{1-{\alpha}}\bigo(\tau^3)+\bigo(\tau^3).
$$
\end{lemma}
\begin{proof}\label{rem : fractional}
We use Theorem~\ref{thm : fractional}, which states that for sufficiently small $\alpha > 0$, $W^{1,p}(\Omega)$ is included in the domain of $(-A)^\alpha$, which does not involve any condition on the boundary. One then obtains
\begin{align*}
\|(-A)^{1+\alpha}\e^{\tau( 1- s)A}\widehat{\psi}(\tau s)\|\leq\|A\e^{\tau(1 - s)A}\|\|(-A)^\alpha\widehat{\psi}(\tau s)\|\leq C\frac{1}{\tau(1-s)}
\end{align*}
and
$$
\|2(-A)^{\alpha}\e^{\tau(1 - s)A}\widehat{\psi}'(\tau s)\|\leq\|2\e^{\tau(1 - s)A}\|\|(-A)^{\alpha},\widehat{\psi}'(\tau s)\|\leq C
$$
which gives
\begin{align*}
\|P_1(\tau s)\| = \|(-A)^{1+\alpha}\e^{\tau(1 - s)A}\widehat{\psi}(\tau s)-2(-A)^{\alpha}\e^{\tau(1 - s)A}\widehat{\psi}'(\tau s)\|
\leq C\left(1+\frac{1}{\tau(1-s)}\right)
\end{align*}
instead of~\eqref{eq : P1}. Similarly, if condition~\eqref{eq:QuadratureCondition} is satisfied, one obtains
\begin{align*}
\|P_1(\tau s)\| = \|(-A)^{1+\alpha}\e^{\tau(1 - s)A}\widehat{\psi}(\tau s)-2(-A)^{\alpha}\e^{\tau(1 - s)A}\widehat{\psi}'(\tau s)\|
\leq C\left(1+\frac{1}{1-s}\right)
\end{align*}
instead of~\eqref{eq : P1 modified}, and this concludes the proof.
\end{proof}

\subsection{Order one error estimate for the five parts Strang splitting}
In this section, we prove that the five parts modified splitting~\eqref{eq : five parts Strang splitting} is of global order one because this is needed in the proof of the global order of the method. We start to give two estimations for the local error.
To perform our convergence analysis, we need an exact formula for~\eqref{eq : five parts Strang splitting}. We expand each flow that appears in the Strang splitting,
\begin{align*}
w_{n}&=\phi^f_{\frac{\tau}{2}}(u(t_n))= u(t_n)+\int_0^{\tau/2}f(\phi^{f}_{s}(u(t_n)))ds,\\
\tilde{w}_n&=\phi^{-q_n}_{\frac{\tau}{2}}(w_n)=u(t_n)+\int_0^{\tau/2}f(\phi^{f}_{s}(u(t_n)))-q_nds,\\
v_n&=\phi^{D+q_n}_{\frac{\tau}{2}}(\tilde{w}_n)=z_n(\tau)\!+\!\e^{\tau A}(\tilde{w}_{n}-z_n(0)) \!+\!\int_0^\tau \!\! \e^{(\tau - s)A}(q_n+Dz_n(s)-\partial_t z_n(s))ds, \\
\tilde{v}_n&=\phi^{-q_n}_{\frac{\tau}{2}}(v_n)=v_n-\int_0^{\tau/2}q_nds,\\
u_{n+1}&=\phi^{f}_{\frac{\tau}{2}}(\tilde{v}_n)=v_n+\int_0^{\tau/2}f(\phi^{f}_{s}(\tilde{v}_n))-q_nds.
\end{align*} 
We obtain the following exact formula for the numerical flow:
\begin{align}\label{eq : numerical solution}
S_\tau(u(t_n))&=z_n(\tau)+\e^{\tau A}\left(u(t_n)+\int_0^{\tau/2}\left(f(\phi^{f}_{s}(u(t_n)))-q_n\right)ds-z_n(0)\right)\nonumber\\
&+\int_0^\tau\e^{(\tau - s)A}(q_n+Dz_n(s)-\partial_t z_n(s))ds 
+\int_0^{\tau/2}\left(f(\phi^{f}_{s}(\tilde{v}_n))-q_n\right)ds.
\end{align}
We define the local error at time $t_{n+1}$, $\delta_{n+1}$, as follows:
$$
\delta_{n+1}:=\mathcal{S}_\tau u(t_n)-u(t_{n+1}).
$$
Using the formula~\eqref{eq:exact solution} of the exact solution and formula~\eqref{eq : numerical solution} of the numerical solution we obtain
\begin{align}\label{eq : local error}
&\delta_{n+1}=\e^{\tau A}\frac{1}{2}\int_0^{\tau}f(\phi^{f}_{s/2}(u(t_n)))-q_nds\nonumber\\
&+\int_0^\tau\e^{(\tau - s)A}(q_n-f(u(t_n+s)))ds+\frac{1}{2}\int_0^{\tau}f(\phi^{f}_{s/2}(\tilde{v}_n))-q_nds.
\end{align}
Since all the integrands are uniformly bounded on $[0,\tau]$, we obtain the following result, which states that the five parts modified Strang splitting~\eqref{eq : five parts Strang splitting} is locally of first order.
\begin{lemma}\label{thm : local error 0}
Under the assumption of section~\ref{section : framework}, the five parts modified Strang splitting~\eqref{eq : five parts Strang splitting} satisfies the following local error estimate: 
$$
\delta_{n+1}=\bigo(\tau).
$$
\end{lemma}
We prove the next local error estimate we use in the theorem for the global error.
\begin{lemma}\label{thm : local error 1}
Under the assumption of section~\ref{section : framework}, the five parts modified Strang splitting~\eqref{eq : five parts Strang splitting} satisfies the following local error estimate,
$$
\delta_{n+1}=A\bigo(\tau^2)+\bigo(\tau^2).
$$
\end{lemma}
\begin{proof}
In formula \eqref{eq : local error} of the local error, we use the trapezoidal quadrature formula to approximate the integrals. By Lemma~\ref{lemma : Quadrature Error lemma 2}, the quadrature error made to approximate $\int_0^\tau\e^{(\tau - s)A}(q_n-f(u(t_n+s)))ds$ is equal to $A\bigo(\tau^2)+\bigo(\tau^2)$.
We get
\begin{align*}
&\delta_{n+1}=\e^{\tau A}\frac{\tau}{4} \left( f(u(t_n))-q_n+f(\phi^{f}_{\tau/2}(u(t_n)))-q_n\right)\\ 
&+\frac{\tau}{2}\left(\e^{\tau A}(q_n-f(u(t_n)))+q_n-f(u(t_n+\tau)))\right)\\
&+\frac{\tau}{4} \left( f(\tilde{v}_n)-q_n+f(\phi^{f}_{\tau/2}(\tilde{v}_n))-q_n\right)+A\bigo(\tau^2)+\bigo(\tau^2).
\end{align*}
Since $\phi^{f}_{\tau/2}(u(t_n))=u(t_n)+\bigo(\tau)$, $f(u(t_n+\tau))=f(u(t_n))+\bigo(\tau)$ and $\tilde{v}_n=u(t_n)+\bigo(\tau)$, and expending $\e^{\tau A}=Id+\tau A \varphi_1(\tau A)$, we obtain
\begin{align*}
&\delta_{n+1}=\frac{\tau}{2} \left( f(u(t_n))-q_n\right) +\tau(q_n-f(u(t_n)))+\frac{\tau}{2} \left( f(u(t_n))-q_n\right)\\
&+\frac{\tau^2}{2}A\varphi_1(\tau A)\left( f(u(t_n))-q_n\right)+\frac{\tau^2}{2}A\varphi_1(\tau A)(q_n-f(u(t_n))+A\bigo(\tau^2)+\bigo(\tau^2)\\
&=A\bigo(\tau^2)+\bigo(\tau^2),
\end{align*}
which concludes the proof.
\end{proof}
Using Theorem~\ref{thm : fractional} and Lemma~\ref{lemma : Quadrature Error lemma fractional}, we can improve Lemma~\ref{thm : local error 1} as follows.
\begin{lemma}\label{lemma : local order fractional 1}
Under the assumption of section~\ref{section : framework}, the five parts modified Strang splitting~\eqref{eq : five parts Strang splitting} satisfies the following. There exists $\alpha>0$ such that
$$
\delta_{n+1}=(-A)^{1-{\alpha}}\bigo(\tau^2)+\bigo(\tau^2).
$$
\end{lemma}
\begin{proof}
Using Lemma~\ref{lemma : Quadrature Error lemma fractional}, we can obtain that the quadrature error made to approximate $\int_0^\tau\e^{(\tau - s)A}(q_n-f(u(t_n+s)))ds$ is equal to $(-A)^{1-{\alpha}}\bigo(\tau^2)+\bigo(\tau^2)$.
We then use Theorem~\ref{thm : fractional} to bound $\frac{\tau^2}{2}A\varphi_1(\tau A)\left( f(u(t_n))-q_n\right)+\frac{\tau^2}{2}A\varphi_1(\tau A)(q_n-f(u(t_n))$. We obtain
\begin{align*}
\|A\varphi_1(\tau A)\left( f(u(t_n))-q_n\right)\|\leq\|(-A)^{1-\alpha}\varphi_1(\tau A)\|\|(-A)^{\alpha} (f(u(t_n))-q_n)\|.
\end{align*}
This gives us the desired result.
\end{proof}
Using the previous results for the local error, we can now prove the following order estimation for the global error.
\begin{proposition}\label{thm : global error one log}
Under the assumptions of section~\ref{section : framework}, the five parts modified Strang splitting~\eqref{eq : five parts Strang splitting} satisfies 
$$\|u_n-u(t_n)\|\leq C\tau(1+|\log\tau|) \quad 0\leq n\tau \leq T.$$
The constant C depends on T but is independent on $\tau$ and n.
\end{proposition}
\begin{proof}
The global error is defined as $e_n=u_n-u(t_n)$. 
$$
e_{n+1}=\mathcal{S}_{\tau}u_n-u(t_{n+1})=\mathcal{S}_{\tau}u_n-\mathcal{S}_{\tau}u(t_n)+\mathcal{S}_{\tau}u(t_n)-u(t_{n+1})=\mathcal{S}_{\tau}u_n-\mathcal{S}_{\tau}u(t_n)+\delta_{n+1}.
$$
Using the exact formula~\eqref{eq : numerical solution} for $\mathcal{S}_{\tau}u_n$ and $\mathcal{S}_{\tau}u(t_n)$, we obtain for $e_{n+1}$,
\begin{align*}
e_{n+1}&=\e^{\tau A} e_n + E(u_n,u(t_n))+\delta_{n+1},
\end{align*}
with 
\begin{align}\label{eq : E} 
E(u_n,u(t_n))&=\e^{\tau A}\int_0^{\frac{\tau}{2}}(f(\phi_s^f(u_n))-f(\phi_s^f(u(t_n))))ds\nonumber\\ 
&+\int_0^{\frac{\tau}{2}}f(\phi_{s}^f\circ\phi_{\frac{\tau}{2}}^{-q_n}\circ\phi_\tau^{D+q_n}\circ\phi_{\frac{\tau}{2}}^{-q_n}\circ\phi_{\frac{\tau}{2}}^f(u_n))ds\nonumber\\
&-\int_0^{\frac{\tau}{2}}f(\phi_{s}^f\circ\phi_{\frac{\tau}{2}}^{-q_n}\circ\phi_\tau^{D+q_n}\circ\phi_{\frac{\tau}{2}}^{-q_n}\circ\phi_{\frac{\tau}{2}}^f(u(t_n)) ) ds.
\end{align}
Let us bound $E(u(t_n),u_n)$. We use the Lipschitz continuity of $f$ and $\phi_s^f$.
For the first integral in~\eqref{eq : E}, we have
\begin{align*}
\left\|\int_0^{\frac{\tau}{2}}(f(\phi_s^f(u_n))-f(\phi_s^f(u(t_n))))ds\right\|&\leq C\tau \|e_n\|.
\end{align*}
For the second integral that appears in~\eqref{eq : E}, we observe that
$$
\phi_s^{q_n}(u)-\phi_s^{q_n}(v)=u-sq_n-v+sq_n=u-v.
$$
We obtain
\begin{align*}
&\left\|\int_0^{\frac{\tau}{2}}f(\phi^f_s\circ\phi_\frac{\tau}{2}^{-q_n}\circ\phi_\tau^{D+q_n}\circ\phi_{\frac{\tau}{2}}^{-q_n}\circ\phi_{\frac{\tau}{2}}^f(u_n)))-f(\phi^f_s\circ\phi_\frac{\tau}{2}^{-q_n}\circ\phi_\tau^{D+q_n}\circ\phi_{\frac{\tau}{2}}^{-q_n}\circ\phi_{\frac{\tau}{2}}^f(u(t_n) ) ds\right\|\\
&\leq C\tau\left\|\phi_\tau^{D+q_n}\circ\phi_{\frac{\tau}{2}}^{-q_n}\circ\phi_{\frac{\tau}{2}}^f(u_n)-\phi_\tau^{D+q_n}\circ\phi_{\frac{\tau}{2}}^{-q_n}\circ\phi_{\frac{\tau}{2}}^f(u(t_n)) \right\|.
\end{align*}
Writing the exact formula for $\phi_\tau^{D+q_n}\circ\phi_{\frac{\tau}{2}}^{-q_n}\circ\phi_{\frac{\tau}{2}}^f$, we have
\begin{align*}
&C\tau\left\|\phi_\tau^{D+q_n}\circ\phi_{\frac{\tau}{2}}^{-q_n}\circ\phi_{\frac{\tau}{2}}^f(u_n)-\phi_\tau^{D+q_n}\circ\phi_{\frac{\tau}{2}}^{-q_n}\circ\phi_{\frac{\tau}{2}}^f(u(t_n))\right\|\\
&=C\tau\left\| \e^{\tau A}(u_n-u(t_n))+\e^{\tau A}\int_0^{\tau/2}f(\phi^{f}_{s}(u_n))-f(\phi^{f}_{s}(u(t_n))))ds \right\|\\
&\leq C\tau \|e_n\|.
\end{align*}
Therefore
$$
\|E(u_n,u
(t_n))\|\leq C\tau \|e_n\|.
$$
The global error $e_n$ satisfies the following recursive formula:
$$
e_n=\e^{n\tau A}e_0+\sum_{k=0}^{n-1}\e^{(n-k-1)\tau A}\delta_{k+1}+\sum_{k=0}^{n-1}\e^{(n-k-1)\tau A}E(u_k,u(t_k)).
$$
This gives us, thanks to to previous estimation for $\|E(u_k,u(t_k))\|$ and $\|\delta_{k}\|$ and since $\|e_0\|=0$,
\begin{align*}
&\|e_n\|\leq \|\e^{n\tau A}\|\|e_0\|+\sum_{k=0}^{n-1}\|\e^{(n-k-1)\tau A}\delta_{k+1}\|+\sum_{k=0}^{n-1}\|\e^{(n-k-1)\tau A}\|\|E(u_k,u(t_k))\|\\
&\leq \sum_{k=0}^{n-2}\|\e^{(n-k-1)\tau A}(A\bigo(\tau^2)+\bigo(\tau^2))\|+\|\delta_n\|+C\tau\sum_{k=0}^{n-1}\|e_k\|.
\end{align*}
Since, by Lemma~\ref{thm : local error 0}, $\|\delta_n\|=\bigo(\tau)$ and using the parabolic smoothing property, we get
\begin{align*}
&\|e_n\|\leq C\tau^2\sum_{k=0}^{n-2}\frac{1}{(n-k-1)\tau}+nC\tau^2+C\tau+C\tau\sum_{k=0}^{n-1}\|e_k\|.
\end{align*}
We rearrange the second sum and observe that $nC\tau^2=C\tau $, which gives
\begin{align*}
&\|e_n\|\leq C\tau\sum_{k=0}^{n-1}\|e_k\|+C\tau^2\sum_{k=1}^{n-1}\frac{1}{k\tau}+C\tau.
\end{align*}
The second sum can be bounded as,
\begin{align*}
C\tau^2\sum_{k=1}^{n-1}\frac{1}{k\tau}&\leq C\tau\int_\tau^{n\tau}\frac{1}{x}dx\leq C\tau\int_\tau^{T}\frac{1}{x}dx
\leq C\tau(1+|\log(\tau)|).
\end{align*}
Using the discrete Gronwall's lemma we obtain the desired result:
\begin{align*}
\|e_n\|&\leq  C\tau(1+|\log(\tau)|)\e^{(n-1)C\tau}\leq C\tau(1+|\log(\tau)|)\e^{CT}=C\tau(1+|\log(\tau)|),
\end{align*}
which concludes the proof.
\end{proof}
Using Theorem~\ref{thm : fractional} and Proposition~\ref{thm : global error one log} we are now in position to prove Proposition~\ref{thm : global order one}, which provides a first order estimation for the global error.
\begin{proof}[Proof of Proposition~\ref{thm : global order one}]
We use Lemma~\ref{lemma : Quadrature Error lemma fractional} to remove the term $\log(\tau)$ in the global error estimate. Indeed, we obtain, 
\begin{align*}
&\|e_n\|\leq C\tau\sum_{k=0}^{n-1}\|e_k\|+C\tau^2\sum_{k=1}^{n-1}\frac{1}{k\tau^{1-\alpha}}+C\tau.
\end{align*}
We then estimate 
\begin{align*}
C\tau^2\sum_{k=1}^{n-1}\frac{1}{k\tau^{1-\alpha}}&\leq C\tau\int_\tau^{n\tau}\frac{1}{x^{1-\alpha}}dx\leq C\tau\int_\tau^{T}\frac{1}{x^{1-\alpha}}dx=C\tau(T^\alpha-\tau^{\alpha})\leq C \tau,
\end{align*}
which concludes the proof.
\end{proof}
\begin{remark}
In Lemma~\ref{lemma : projection 1}, to show that a function $q_n$ satisfying the boundary condition~\eqref{eq : q condition} or~\eqref{eq : q condition 2} satisfies the condition~\eqref{eq:QuadratureCondition}, we need to use Proposition~\ref{thm : global order one}, which we prove using Theorem~\ref{thm : fractional}. To prove $\|u_n-u(t_n)\|\leq C\tau^2(1+|\log\tau|)$ without using Theorem~\ref{thm : fractional}, we need a weaker condition that $q_n$ must satisfy, for example, $f(u(t_n+s))-q_n=\phi_n+\bigo(\tau(1+|\log(\tau)|))$ with $\phi_n\in D(A)$ and $A\phi_n=\bigo(1)$, instead of~\eqref{eq:QuadratureCondition}. We can then prove, with the help of the bound, $\|u_n-u(t_n)\|\leq C\tau(1+|\log\tau|)$, that a function satisfying~\eqref{eq : q condition} satisfies this new condition. We decide not to follow this approach as we think it simplifies our arguments to only have condition~\eqref{eq:QuadratureCondition} throughout the paper. 
\end{remark}
\subsection{Analysis of the corrector function}
We show that the conditions~\eqref{eq : q condition} and~\eqref{eq : q condition 2} for $q_n$ are properly chosen, that is $\widehat{\psi}(s)=q_n-f(u(t_n+s))$ satisfies the hypothesis~\eqref{eq:QuadratureCondition} when one of those conditions is satisfied.
For that purpose, we use Proposition~\ref{thm : global order one}, which states that $u(t_n)-u_n=\bigo(\tau)$. We stress that no condition on $q_n$ is required in the proof of Proposition~\ref{thm : global order one}. We first consider the boundary condition~\eqref{eq : q condition} for $q_n$. 
\begin{lemma}\label{lemma : projection 1}
Let $q_n$ be chosen such that~\eqref{eq : q condition} is satisfied,
\begin{align*}
Bq_n&=\frac{2}{\tau}(B\phi^f_{\tau/2}(u_n)-B(u_n)).
\end{align*}
Let $\widehat{\psi}(s)=q_n-f(u(t_n+s))$.
Then $\widehat{\psi}(s)=\phi_n+\bigo(\tau)$ with $\phi_n\in D(A)$ and $A\phi_n=\bigo(1)$.
\end{lemma}
\begin{proof}
We observe that 
$$f(u_n)=\frac{2}{\tau}\phi^f_{\frac{\tau}{2}}(u_n)-\frac{2}{\tau}u_n+\bigo(\tau).$$
We obtain, with Proposition~\ref{thm : global order one}, that
\begin{align*}
f(u(t_n+s))&=f(u_n)+\bigo(\tau)=\frac{2}{\tau}\phi^f_{\frac{\tau}{2}}(u_n)-\frac{2}{\tau}u_n+\bigo(\tau).
\end{align*}
We define $\phi_n$ as follows,  
$$
\phi_n=q_n-\left(\frac{2}{\tau}\phi^f_{\frac{\tau}{2}}(u_n)-\frac{2}{\tau}u_n\right).
$$
Since $B(q_n-(\frac{2}{\tau}\phi^f_{\frac{\tau}{2}}(u_n)-\frac{2}{\tau}u_n))=0$, $\phi_n$ is in $D(A)$.
Furthermore, a simple calculation yields
$$
\|A\phi_n\|\leq C.
$$
Therefore $A\phi_n=\bigo(1)$.
\end{proof}
We now consider the boundary condition~\eqref{eq : q condition 2} for the corrector functions.
\begin{lemma}\label{lemma : projection 2}
Let $q_n$ be chosen such that~\eqref{eq : q condition 2} is satisfied,
\begin{align*}
Bq_n&=\frac{2}{\tau}(B\phi^f_{\tau/2}(u_n)-b_n).
\end{align*}
Let $\widehat{\psi}(s)=q_n-f(u(t_n+s))$.
Then $\widehat{\psi}(s)=\phi_n+\bigo(\tau)$, with $\phi_n\in D(A)$ and $A\phi_n=\bigo(1)$. 
\end{lemma}
\begin{proof}
The proof is conducted by induction. Since $Bu_0=b_0$, we know by Lemma~\ref{lemma : projection 1} that the result is true for $n=0$.

We assume that the statement is true for $n=0,\ldots,k-1$. Let us show that it is true for $n=k$. We write the exact formula for $\phi^f_{\frac{\tau}{2}}(u_k)$ in function of $v_{k-1}$ and $q_k$:
\begin{align*}
\phi^f_{\frac{\tau}{2}}(u_k)&=v_{k-1}-\frac{\tau}{2}q_{k-1}+\int_0^\tau f(\phi^f_s(v_{k-1}-\frac{\tau}{2}q_{k-1}))ds.
\end{align*}
We then apply the midpoint quadrature formula to the integral and obtain an error of size $\bigo(\tau^2)$ since $f$ is twice continuously differentiable:
\begin{align*}
\phi^f_{\frac{\tau}{2}}(u_k)&=v_{k-1}+\tau f(\phi^f_{\frac{\tau}{2}}(v_{k-1}-\frac{\tau}{2}q_{k-1}))-\frac{\tau}{2}q_{k-1}+\bigo(\tau^2).
\end{align*}
We observe that $\phi^f_{\frac{\tau}{2}}(v_{k-1}-\frac{\tau}{2}q_{k-1})=u_k$. Since by Proposition~\ref{thm : global order one}, $u_k=f(u(t_k))+\bigo(\tau)$ and since $f(u(t_k))=f(u(t_{k-1}))+\bigo(\tau)$, we obtain
\begin{align*}
\phi^f_{\frac{\tau}{2}}(u_k)&=v_{k-1}+\frac{\tau}{2} f(u(t_k))+\frac{\tau}{2} f(u(t_{k-1}))-\frac{\tau}{2}q_{k-1}+\bigo(\tau).
\end{align*}
Since $Bv_{k-1}=b_k$, and since by hypothesis $f(u(t_{k-1}))-q_{k-1}=\phi_{k-1}+\bigo(\tau)$ with $\phi_{k-1}\in D(A)$ and $A\phi_{k-1}=\bigo(1)$, we obtain
$$
Bq_k\!=\!\frac{2}{\tau}(Bv_{k-1}-b_k)+Bf(u(t_k))+Bf(u(t_{k-1}))- Bq_{k-1}+B\bigo(\tau)\!=\!Bf(u(t_k))+B\bigo(\tau).
$$
We can therefore decompose $q_k$ as $q_k=\tilde{q}_k+r_k$, where $B\tilde{q}_k=Bf(u(t_k))$ and $r_k=\bigo(\tau)$ and choose $\phi_k=\tilde{q}_k-f(u(t_k))$. Then $B\phi_k=0$ and a simple calculation yields $\|A\phi_k\|\leq C$. This concludes the proof.
\end{proof}
\subsection{Order two error estimate for the five parts modified Strang splitting}

The following lemma is an estimation of the local error for the modified Stang splitting, which states that the five parts modified Strang splitting~\eqref{eq : five parts Strang splitting} is locally a method of second order.
\begin{lemma}\label{lemme : Local order estimate}
Under the assumption of section~\ref{section : framework}, the five parts modified Strang splitting~\eqref{eq : five parts Strang splitting} satisfies  
$$
\delta_{n+1}=\bigo(\tau^2).
$$
\end{lemma}
\begin{proof}

In the exact formula of $\delta_{n+1}$ \eqref{eq : local error}, we use the left rectangle quadrature formula for the first and third integral and a right quadrature formula for the second integral. 
By Lemma~\ref{lemma : Quadrature Error lemma}, the quadrature error is $\bigo(\tau^2)$.
We get
\begin{align*}
\delta_{n+1}=\e^{\tau A}\frac{\tau}{2}(f(u(t_n))-q_n)+\tau(q_n-f(u(t_n+\tau)))+\frac{\tau}{2}(f(\tilde{v}_n)-q_n)+\bigo(\tau^2).
\end{align*}
Expending $\e^{\tau A}$ and $u(t_n)$ around $\tau$, that is, $\e^{\tau A}\frac{\tau}{2}(f(u(t_n))-q_n)=\frac{\tau}{2}(f(u(t_{n+1}))-q_n)+\bigo(\tau^2)$, we obtain the following result:
\begin{align*}
\delta_{n+1}=\frac{\tau}{2}(f(\tilde{v}_n)-f(u(t_{n+1}))+\bigo(\tau^2).
\end{align*}
We use the exact formula for $\tilde{v}_n$ and $u(t_{n+1})$ and the Lipschitz continuity of $f$:
\begin{align*}
\|\delta_{n+1}\|&\leq C\frac{\tau}{2} \Big\|\e^{\tau A}\int_0^{\tau/2}(f(\phi^{f}_{s}(u(t_n)))-q_n)ds+\int_0^\tau\e^{(\tau - s)A}(q_n-f(u(t_n+s)))ds\Big\|\\
&+\bigo(\tau^2).
\end{align*}
Since the integrands are uniformly bounded on $[0,\tau]$, we get $
\delta_{n+1}=\bigo(\tau^2)$, which concludes the proof.
\end{proof}
The following lemma gives the second local error estimates that we need in the proof for the global convergence of the method.

\begin{lemma}\label{thm : local order modified 2}
Under the assumption of section~\ref{section : framework}, the five parts modified Strang splitting~\eqref{eq : five parts Strang splitting} satisfies
$$
\delta_{n+1}=A\bigo(\tau^3)+\bigo(\tau^3).
$$
\end{lemma}
\begin{proof}
We start as in the proof of Lemma~\ref{thm : local error 1},  by using trapezoidal quadrature formulas to approximate the integrals in formula \eqref{eq : local error} of the local error. By Lemma~\ref{lemma : Quadrature Error lemma 2}, the quadrature error made to approximate $\int_0^\tau\e^{(\tau - s)A}(q_n-f(u(t_n+s)))ds$ is equal to $A\bigo(\tau^3)+\bigo(\tau^3)$.
We get
\begin{align*}
&\delta_{n+1}=\e^{\tau A}\frac{\tau}{4} \left( f(u(t_n))-q_n+f(\phi^{f}_{\tau/2}(u(t_n)))-q_n\right)\\
&+\frac{\tau}{2}\left(\e^{\tau A}(q_n-f(u(t_n)))+q_n-f(u(t_n+\tau)))\right)\\
&+\frac{\tau}{4} \left( f(\tilde{v}_n)-q_n+f(\phi^{f}_{\tau/2}(\tilde{v}_n))-q_n\right)+A\bigo(\tau^3)+\bigo(\tau^3)
\end{align*}
By Lemma~\ref{lemme : Local order estimate}, we have the following equality:
\begin{align}\label{eq : local error bis}
f(\phi_{\tau/2}^{f}(\tilde{v}_n))&=f(\mathcal{S}_{\tau}(u(t_n)))=f(u(t_n+\tau))+\bigo(\tau^2).
\end{align}
We obtain
\begin{align*}
&\delta_{n+1}=\e^{\tau A}\frac{\tau}{4}(f(\phi^f_{\tau/2}(u(t_n)))-q_n)-\frac{\tau}{4}\e^{\tau A}(f(u(t_n))-q_n)\\
&+\frac{\tau}{4}(f(\tilde{v}_n)-q_n)-\frac{\tau}{4}(f(u(t_n+\tau))-q_n)+A\bigo(\tau^3)+\bigo(\tau^3).
\end{align*}
We observe that $\frac{\tau}{2}\e^{\tau A}(f(u(t_n))-q_n)+\frac{\tau}{2}(f(u(t_n+\tau))-q_n)$ is the trapezoidal quadrature formula for $\int_0^\tau \e^{(\tau-s)A}(f(u(t_n+s)))-q_n)ds$ and that $\e^{\tau A}\frac{\tau}{2}(f(\phi^f_{\tau/2}(u(t_n)))-q_n)+\frac{\tau}{2}(f(\tilde{v}_n)-q_n)$ is the trapezoidal quadrature formula for $\int_0^\tau \e^{(\tau-s)A}(f(\phi_{(\tau-s)/2}^{f}\circ\phi^{-q_n}_{s/2}\circ\phi_s^{D+q_n}\circ\phi^{-q_n}_{s/2}\circ\phi_{s/2}^{f}(u(t_n)))-q_n)ds$. Since $\phi_{(\tau-s)/2}^{f}\circ\phi^{-q_n}_{s/2}\circ\phi_s^{D+q_n}\circ\phi^{-q_n}_{s/2}\circ\phi_{s/2}^{f}(u(t_n))=u(t_n)+\bigo(\tau)$, the second integrand satisfies the hypotheses of Lemma~\ref{lemma : Quadrature Error lemma 2}, by Lemmas \ref{lemma : projection 1} and \ref{lemma : projection 2}. Applying Lemma~\ref{lemma : Quadrature Error lemma 2}, we therefore have a quadrature error of the form $A\bigo(\tau^3)+\bigo(\tau^3)$,
\begin{align*}
&\delta_{n+1}=\frac{1}{2}\int_0^\tau \e^{(\tau-s)A}(f(\phi_{\frac{\tau-s}{2}}^{f}\circ\phi^{-q_n}_{\frac{s}{2}}\circ\phi_s^{D+q_n}\circ\phi^{-q_n}_{\frac{s}{2}}\circ\phi_{\frac{s}{2}}^{f}(u(t_n)))-q_n)ds \\
&-\frac{1}{2}\int_0^\tau \e^{(\tau-s)A}(f(u(t_n+s)))-q_n)ds+A\bigo(\tau^3)+\bigo(\tau^3).
\end{align*}
Applying the midpoint quadrature method to both integrals, we obtain
\begin{align*}
&\delta_{n+1}=\frac{\tau}{4}\e^{\frac{\tau}{2} A}(f(\mathcal{S}_{\frac{\tau}{2}}(u(t_n)))-q_n)-\frac{\tau}{4}\e^{\frac{\tau}{2} A}(f(u(t_n+\frac{\tau}{2}))-q_n)+A\bigo(\tau^3)+\bigo(\tau^3).
\end{align*}
Using Lemma~\ref{lemme : Local order estimate}, the Lipschitz continuity of $f$, and the boundedness of $\e^{\frac{\tau}{2} A}$ we have the desired result,
\begin{align*}
\delta_{n+1}&=A\bigo(\tau^3)+\bigo(\tau^3),
\end{align*}
which concludes the proof.
\end{proof}
We now improve Lemma~\ref{thm : local order modified 2} as follows using Lemma~\ref{lemma : Quadrature Error lemma fractional} instead of Lemma~\ref{lemma : Quadrature Error lemma 2} for the quadrature error.
\begin{lemma}\label{thm : local error modified fractional}
Under the assumption of section~\ref{section : framework}, the five parts modified Strang splitting~\eqref{eq : five parts Strang splitting} satisfies the following. There exists $\alpha>0$ such that
$$
\delta_{n+1}=(-A)^{1-{\alpha}}\bigo(\tau^3)+\bigo(\tau^3).
$$
\end{lemma}
Using the previous results for the local error, we can prove Proposition~\ref{thm : global error}. It is the main global error estimate that we obtain without using Theorem~\ref{thm : fractional}. 
\begin{proof}[Proof of Proposition~\ref{thm : global error}]
The proof is similar to the proof of the Proposition~\ref{thm : global error one log}, except for the bounds of the local errors $\delta_{n}$, since if~\eqref{eq : q condition} is satisfied, $\delta_k=A\bigo(\tau^3)+\bigo(\tau^3)$ and $\delta_n=\bigo(\tau^2)$.
Hence
\begin{align*}
\|e_n\|&\leq  C\tau^2(1+|\log(\tau)|)\e^{(n-1)C\tau}\leq C\tau^2(1+|\log(\tau)|)\e^{CT}=C\tau^2(1+|\log(\tau)|),
\end{align*}
which concludes the proof.
\end{proof}
\begin{proof}[Proof of Theorem~\ref{thm : global error fractional}]
We follow the proof of Proposition~\ref{thm : global order one}, using Lemma~\ref{thm : local error modified fractional} to remove the $\log(\tau)$ term in the global error estimate of Proposition~\ref{thm : global error}.
\end{proof}

\section{Numerical experiments}\label{section : Numerical experiments }
In this section we perform several numerical experiments in dimension $d=1,2$ to illustrate the performance of the five parts modified Strang splitting~\eqref{eq : five parts Strang splitting} when applied to diffusion problems with various nonlinearities. The norm we use to compute the numerical error is
$$
\|u\|_{L^{\infty}([0,T],L^2(\Omega))}=\sup_{t\in[0,T]}\|u(t)\|_{L^2(\Omega)},
$$
where we use the trapezoidal rule to approximate the $L^2(\Omega)$ norm. Considering the domain $(0,1)^d$ equipped with an uniform spatial grid with size $\Delta x$, we use the usual second order finite difference approximation of the Laplacian. We compute the diffusion flows in time exactly using the Matlab package described in~\cite{Nie12} for approaching matrix exponentials and the $\varphi_1$ matrix function~\eqref{eq : phi1}.

In the numerical experiments that follow, we denote the classical Strang splitting~\eqref{eq : classical Strang 1} by \emph{Strang}, the modified Strang splitting~\eqref{eq : modified Strang 1} constructed in~\cite{Ost15,Ost16} by \emph{StrangM3}, the five parts modified Strang splitting~\eqref{eq : five parts Strang splitting} with correction~\eqref{eq : classical Strang 1} by \emph{StrangM5a} and the five parts modified Strang splitting~\eqref{eq : five parts Strang splitting} with correction~\eqref{eq : classical Strang 2} by \emph{StrangM5b}.
In Table~\ref{tab : Flows Evaluations}, we recall the number of evaluations of the diffusion flows~(\ref{eq : D},~\ref{eq : D+q}) and the number of evaluations of the source term flows~(\ref{eq : f},~\ref{eq : f-q}) required at each step of the algorithm for the four considered numerical methods. We also collect the total number of flows after $n$ steps, where we consider that the correction steps~\eqref{eq : projection} are negligible. We recall that for \emph{Strang} and \emph{StrangM5b}, only one evaluation of the diffusion flow and one evaluation of the source term flow is needed after the first step, due to the semigroup property of the exact flows (see Remark~\ref{rem : composition}).
\begin{table}[ht]

\centering
\begin{tabular}{|l||c|c|c|}
  \hline
   Splitting methods& Diffusion flows & Source term flows  & Total number of \\
   &for one step & for one step &flows for $n$ steps\\[2pt]
  \hline \hline
  Classical Strang~\eqref{eq : classical Strang 1} & 1 & 1 & $2n+1$ \\ \hline
  StrangM3~\eqref{eq : modified Strang 1}  & 2 & 1 & $3n$ \\
  \hline
   StrangM5a~(\ref{eq : five parts Strang splitting},~\ref{eq : q condition}) & 1 & 2 & $3n$\\
  \hline
   StrangM5b~(\ref{eq : five parts Strang splitting},~\ref{eq : q condition 2})& 1 & 1 & $2n+1$ \\
  \hline
\end{tabular}\caption{\textit{Number of evaluations of the diffusion flows~(\ref{eq : D},~\ref{eq : D+q}) and of the source term flows~(\ref{eq : f},~\ref{eq : f-q}) needed by the considered splitting methods.}}
 \label{tab : Flows Evaluations}
\end{table}

\paragraph*{A quadratic nonlinearity \textnormal{(see Figure~\ref{fig : quadratic})}}First, we apply the splitting methods \emph{Strang}~\eqref{eq : classical Strang 1}, \emph{StrangM3}~\eqref{eq : modified Strang 1}, \emph{StrangM5a}~(\ref{eq : five parts Strang splitting},~\ref{eq : q condition}) and \emph{StrangM5b}~(\ref{eq : five parts Strang splitting},~\ref{eq : q condition 2}) to a problem given in~\cite[Example 5.2]{Ost16}. We then change the nonlinearity to see how the splitting methods behave. The non linearities we consider are $f(u)=u^2$ and  $f(u)=5u^2$. The case $f(u)=u^2$ is the one presented in~\cite[Example 5.2]{Ost16}. We perform the experiment with mixed boundary conditions, $u(0)=1$, $\partial_n u(1)=1$. We choose a smooth initial condition that satisfies the prescribed boundary conditions. We obtain the following equation with $m=1$ and $m=5$. 
\begin{align} \label{eq:Exemple52}
\partial_t u(x,t) &= \partial_{xx}u(x,t)+mu(x,t)^2,\nonumber\\
u(0,t)&=1,\qquad
\partial_n u(1,t) = 1,\nonumber\\
u(x,0)&=1+\frac{2}{\pi}-\frac{2}{\pi}\cos(\frac{1}{2}\pi x).
\end{align}
The correction we use for \emph{StrangM3}, given in~\cite{Ost16} for $m=1$, is $q_n=m+2mxu_n(1)$. We use $500$ spatial points to discretize the interior of $\Omega$. We compute the solution at final time $T=0.1$. The chosen time steps are $\tau=0.02\cdot 2^{-k}$, $k=0,\ldots,6$. The reference solution is computed with the classical fourth order explicit Runge-Kutta method and a very small time step $\tau=0.02\cdot 2^{-14}\approx 10^{-6}$. In the splitting algorithms, we use the analytic formulas for computing the flows $\phi^f_{\frac{\tau}{2}}(u_n)$ and $\phi^{f-q_n}_{\tau}(u_n)$.
\begin{figure}
\centering
\begin{subfigure}{.5\textwidth}
\centering
\includegraphics[width=\textwidth]{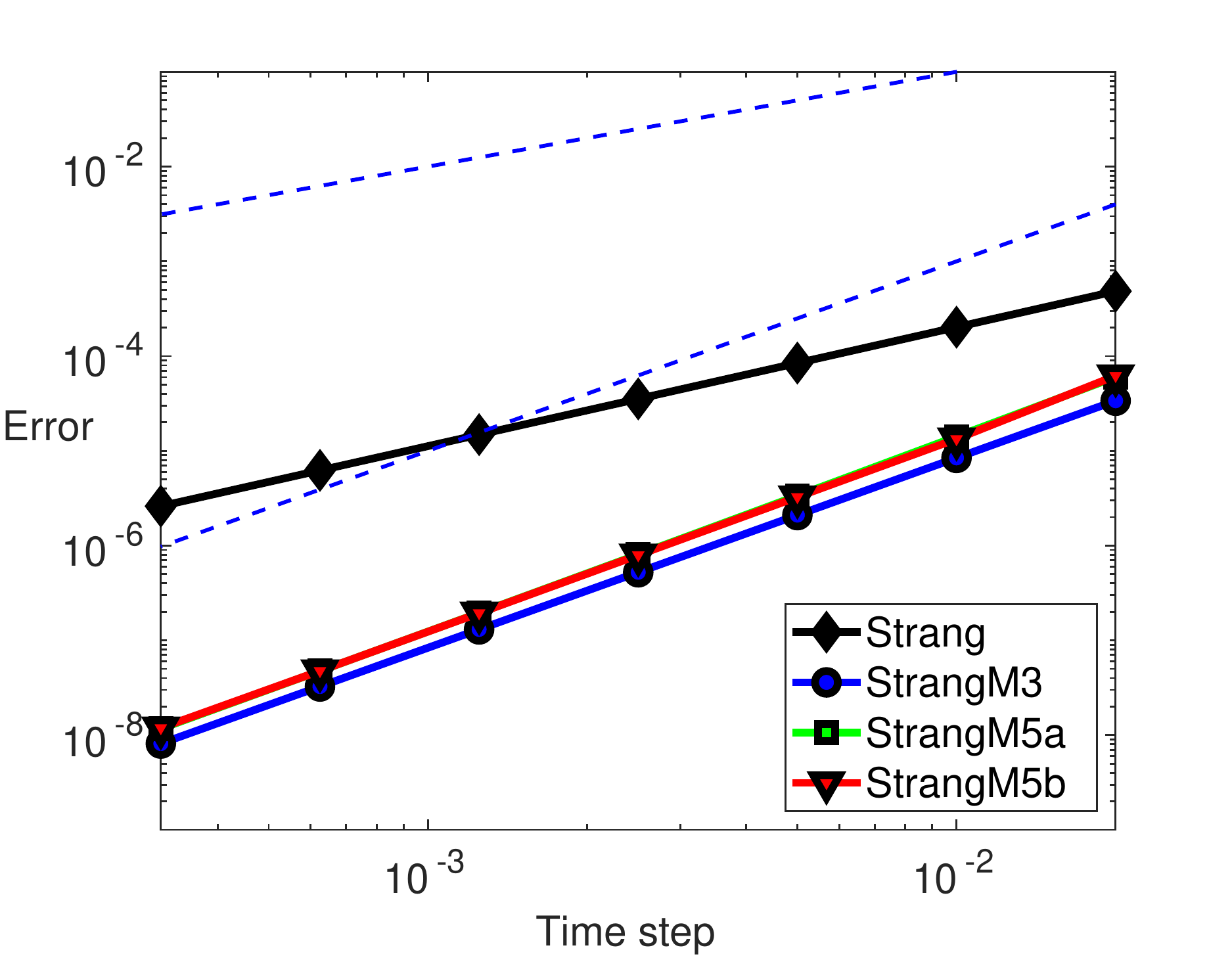} 
\caption*{\textit{$f(u)=u^2$}}
\end{subfigure}%
\begin{subfigure}{.5\textwidth}
\centering
\includegraphics[width=\textwidth]{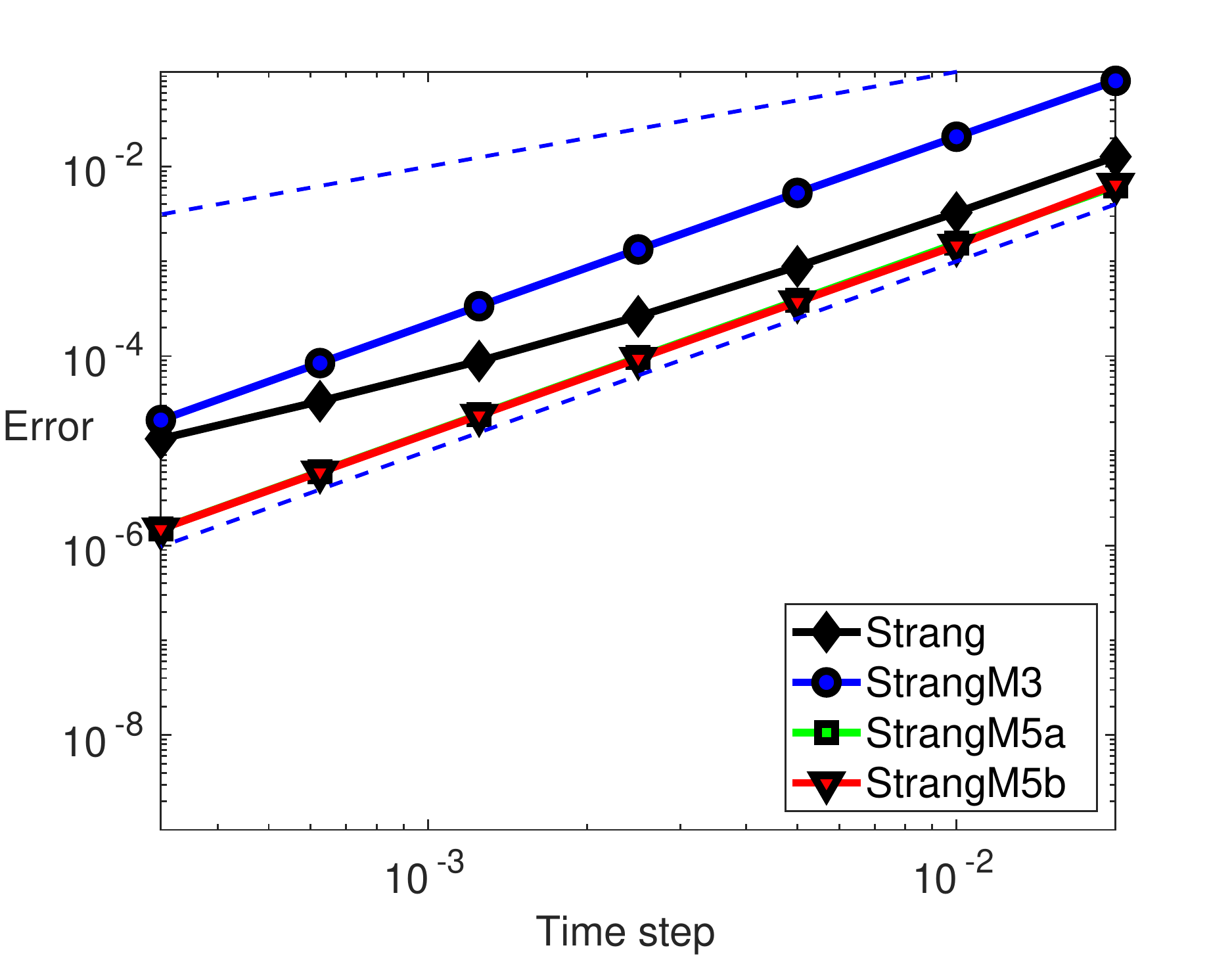}
\caption*{\textit{$f(u)=5u^2$}}
\end{subfigure}
\caption{\textit{Comparison between the splitting methods \emph{Strang}~\eqref{eq : classical Strang 1}, \emph{StrangM3}~\eqref{eq : modified Strang 1}, \emph{StrangM5a}~(\ref{eq : five parts Strang splitting},~\ref{eq : q condition}) and \emph{StrangM5b}~(\ref{eq : five parts Strang splitting},~\ref{eq : q condition 2}) when applied to problem~\eqref{eq:Exemple52} for $m=1$ and $m=5$ on $[0,1]$ with inhomogeneous mixed boundary conditions. The convergence curves of StrangM5a and StrangM5b overlap. Reference slopes one and two are given in dotted lines}.}\label{fig : quadratic}
\end{figure}

We observe in Figure~\ref{fig : quadratic} that the splittings \emph{StrangM3}, \emph{StrangM5a} and \emph{StrangM5b} are all of order two compared to the classical splitting \emph{Strang}. The methods \emph{StrangM5a} and \emph{StrangM5b} have a slightly worse constant of error compared to the splitting \emph{StrangM3} of~\cite{Ost16} for $f(u)=u^2$. Recall that \emph{StrangM3} costs two diffusion flows per time step while \emph{Strang}, \emph{StrangM5a} and  \emph{StrangM5b} cost only  one. For $f(u)=5u^2$, \emph{StrangM5a} and \emph{StrangM5b} become a lot more accurate.
\\
\paragraph*{A meteorology model with an integral source term \textnormal{(see Figure~\ref{fig : IntegroDiff})}} We apply the splitting methods \emph{Strang}~\eqref{eq : classical Strang 1}, \emph{StrangM3}~\eqref{eq : modified Strang 1}, \emph{StrangM5a}~(\ref{eq : five parts Strang splitting},\thinspace\ref{eq : q condition}) and\linebreak \emph{StrangM5b}~(\ref{eq : five parts Strang splitting},~\ref{eq : q condition 2}) to a problem presented in~\cite[Equation 3.1]{Vas92}, where we replace the left Dirichlet boundary condition $u(0,t)=2(2-\sqrt{t})$ by $u(0,t)=2(2-t)$ to have a time continuously differentiable boundary condition. The considered differential equation is the following
\begin{align} \label{eq:IntegroDiff}
\partial_t u(x,t) &= \partial_{xx}u(x,t)-\int_0^1u(s,t)^4\frac{1}{(1+|x-s|)^2}ds,\nonumber\\
u(0,t)&=2(2-t),\qquad
\partial_n u(1,t) = 0,\nonumber\\
u(x,0)&=2(\cos(\pi x)+1),
\end{align}
for $\Omega = [0,1]$. 
We choose $500$ points to discretize the interior of $[0,1]$. We then apply the splitting methods with different time steps $\tau = 2\cdot 10^{-2} \cdot 2^{-k}$, $k=0,\ldots,6$. The reference solution is computed with the classical fourth order explicit Runge-Kutta method and a very small time step $\tau=0.02\cdot 2^{-14}\approx 10^{-6}$. To solve the integral, we use the trapezoidal quadrature formula with the $502$ nodes given by the space discretization. To solve $\partial_t u = f(u)$ and $\partial_t u = f(u)-q_n$, we apply five steps of the classical order four explicit Runge-Kutta method with a time step $\frac{\tau}{10}$. Note that we compute $\partial_t u = f(u)$ and $\partial_t u = f(u)-q_n$ with more precision than is needed in practice to emphasize on the error of the splitting itself. Instead, one can simply use one step of a second order Runge-Kutta method. We compute the solution at final time $T=0.1$.
\begin{figure}
\centering
\begin{subfigure}{.5\textwidth}
\centering
\includegraphics[width=\textwidth]{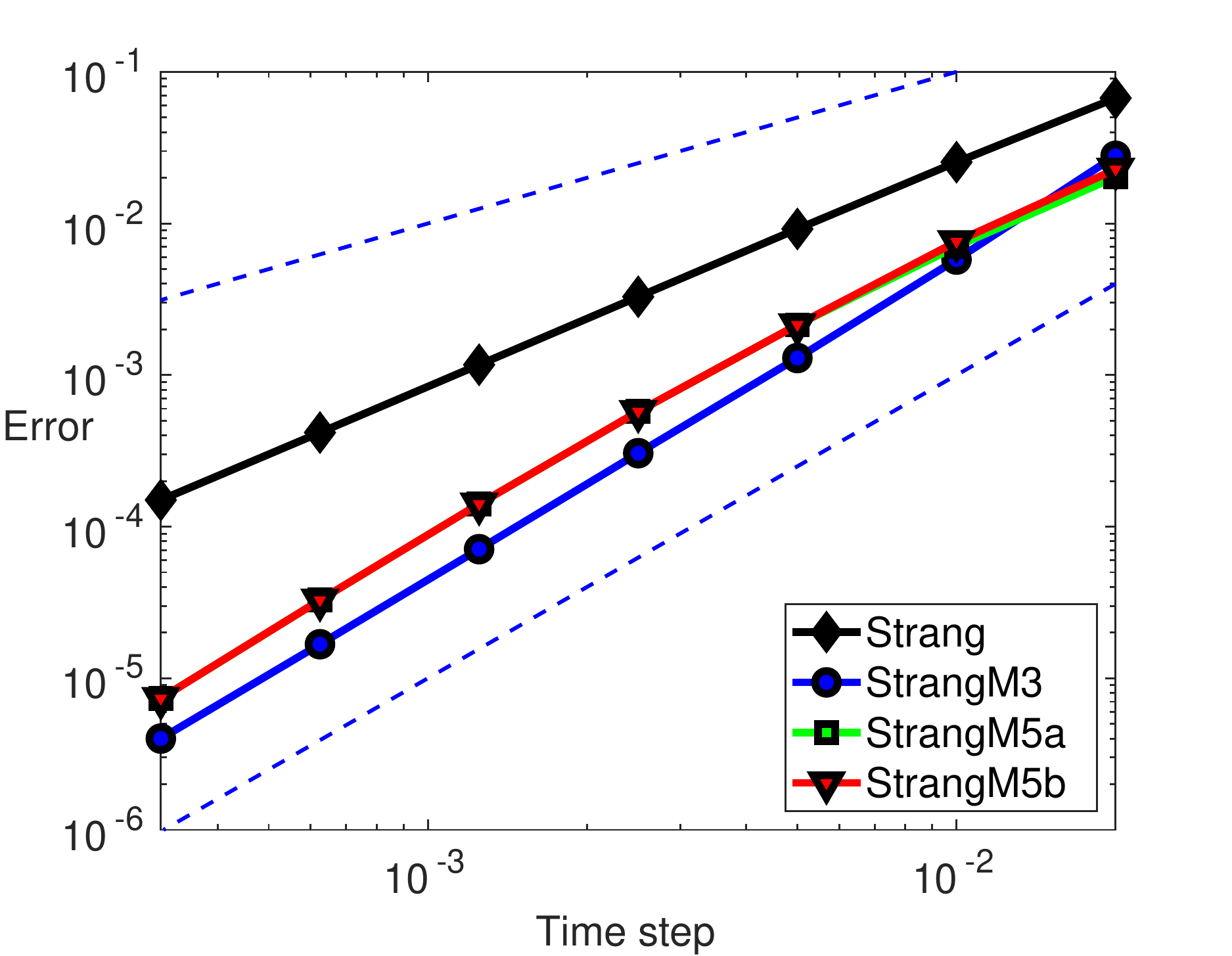} 
\end{subfigure}%
\begin{subfigure}{.5\textwidth}
\centering
\includegraphics[width=\textwidth]{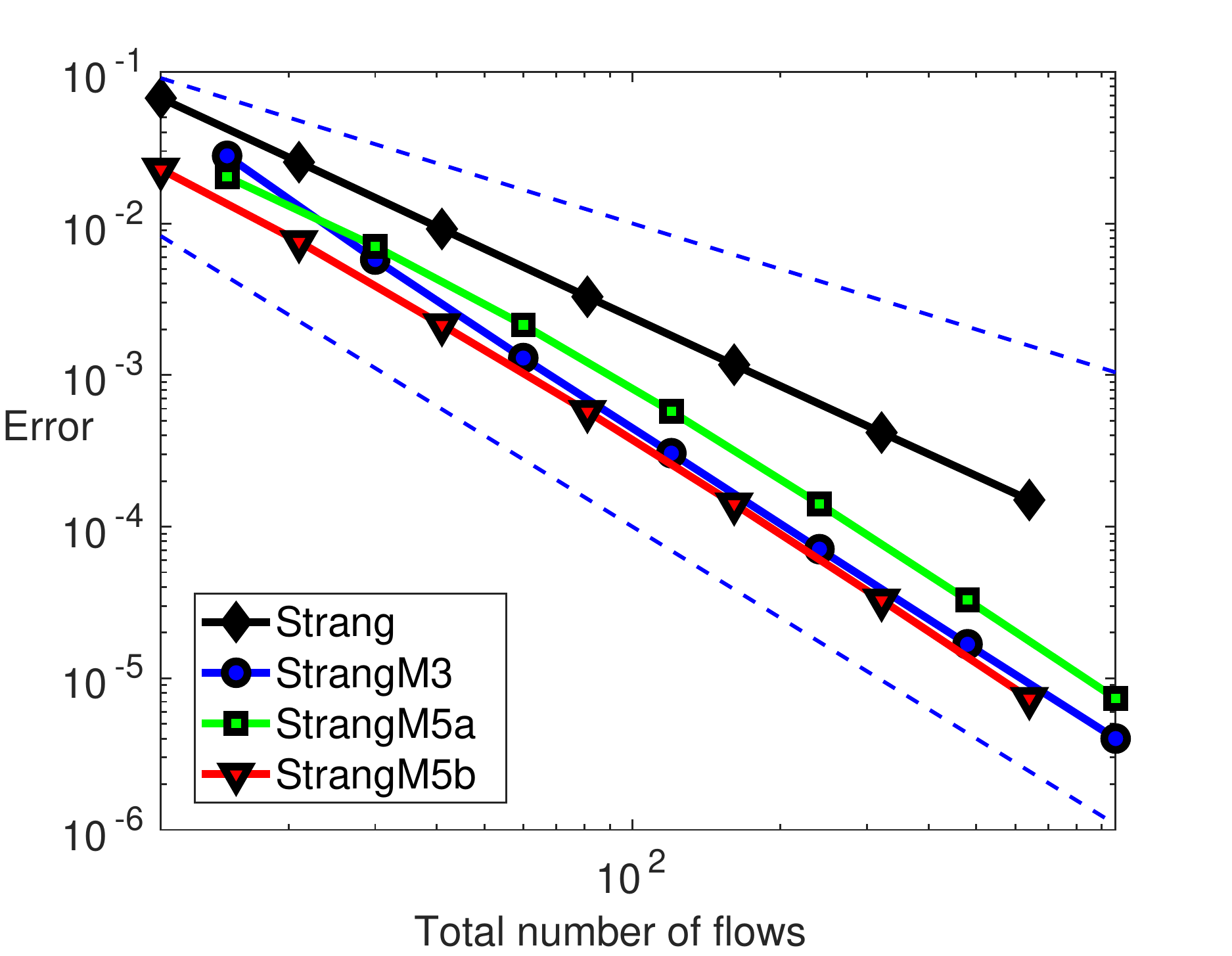}
\end{subfigure}
\caption{\textit{Comparison between the splitting methods \emph{Strang}~\eqref{eq : classical Strang 1}, \emph{StrangM3}~\eqref{eq : modified Strang 1}, \emph{StrangM5a}~(\ref{eq : five parts Strang splitting},~\ref{eq : q condition}) and \emph{StrangM5b}~(\ref{eq : five parts Strang splitting},~\ref{eq : q condition 2}) when applied to the integro-differential equation~\eqref{eq:IntegroDiff} with time dependent boundary conditions. Reference slope ones and two are drawn in dotted line. On the right picture, we compare the error with the number of evaluations of the diffusion flows~(\ref{eq : D},~\ref{eq : D+q}) and the number of evaluations of the source term flows~(\ref{eq : f},~\ref{eq : f-q}) computed during the time integration.}}\label{fig : IntegroDiff}
\end{figure}

We observe in the left picture of Figure~\ref{fig : IntegroDiff} that the modified splitting \emph{StrangM3} given in~\cite{Ost16} has a slightly better constant of error than \emph{StrangM5a} and \emph{StrangM5b}. In the right picture of Figure~\ref{fig : IntegroDiff} however, we recover that the splitting \emph{StrangM5b} requires less evaluations of the flows to obtain a given precision (see Table~\ref{tab : Flows Evaluations} and Remark~\ref{rem : composition}). Since both flows have similar computational costs, the splitting \emph{StrangM5b} is slightly more accurate for the same computational cost. 
Moreover the construction of the correction $q_n$ for the  splitting \emph{StrangM5b} requires less computational cost and is easier to implement than the correction $q_n$ of~\emph{StrangM3} given by~\eqref{eq : correction 3 parts}. Indeed, one needs then to evaluate $f$ on the boundary at each step of the algorithm.

\paragraph*{Case of a stiff nonlinearity \textnormal{(see Figure~\ref{fig : Stiff})}}
 In the next experiments, we compare \emph{Strang}~\eqref{eq : classical Strang 1}, \emph{StrangM3}~\eqref{eq : modified Strang 1}, \emph{StrangM5a}~(\ref{eq : five parts Strang splitting},~\ref{eq : q condition}) and \emph{StrangM5b}~(\ref{eq : five parts Strang splitting},~\ref{eq : q condition 2}) when applied to a stiff problem in a two dimensional domain $\Omega = [0,1]\times [0,1]$. We choose the nonlinearity $f(u)=(1-M\sin(\pi x)\sin(\pi y))u^2$ where $M=1$ in the nonstiff case and $M=100$ in the stiff case. We impose Dirichlet boundary conditions on the left boundary and Neumann boundary conditions on the bottom, top and right boundaries. The boundary conditions are chosen to be consistent with the initial condition $u_0=\frac{\e^x + \e^y}{2}$. We obtain the following problem 
\begin{align}
\partial_t u(x,y,t) &= \partial_{xx} u(x,y,t)+\partial_{yy}u(x,y,t)+(1-M\sin(\pi x)\sin(\pi y))u(x,y,t)^2,\nonumber\\
u(0,y,t)&=\frac{1+\e^y}{2}, \quad
\partial_n u(1,y,t)=\frac{\e}{2}, \quad
\partial_n u(x,0,t) = -\frac{1}{2}, \quad
 \partial_n u(x,1,t) = \frac{\e}{2},                          \nonumber\\
u(x,y,0)&=\frac{\e^x + \e^y}{2},
\end{align}\label{eq : stiff}
where $x,y\in\Omega$, $t\in[0,T]$ and with $M=1$ or $M=100$.
The correction functions $q_n$ are constructed with a multigrid iterative smoothing algorithm, analogously to~\cite[Algorithm 1]{Ost16}. We recall that the method \emph{StrangM3} in~\cite{Ost16} was proposed for nonstiff nonlinearities. Indeed note that $q_n$ is of size $\bigo(M)$ and becomes large for the stiff case $M=100$. For instance in dimension one, on the domain $[0,1]$ with $f(u)=(1-M\sin(\pi x))u^2$ and boundary conditions $u(0,t)=1$, $\partial_nu (1,t)=1$, one obtains $q_n=1+(M\pi u_n(1)^2+2u_n(1))x$. We use a mesh with size $\Delta x = \frac{1}{128}$ to discretize the interior of $\Omega$. We compute the solution at final time $T=0.1$. The chosen time steps are $\tau=0.1\cdot 2^{-k}$, $k=0,\ldots,8$. The reference solution is computed with the classical fourth order explicit Runge-Kutta method and a really small time step $\tau=0.1\cdot 2^{-14}\approx 10^{-6}$. In the splitting algorithms, we use the analytic formulas of $\phi^f_{\frac{\tau}{2}}(u_n)$ and $\phi^{f-q_n}_{\tau}(u_n)$ .
\begin{figure}
\centering
\begin{subfigure}{.5\textwidth}
\centering
\includegraphics[width=\textwidth]{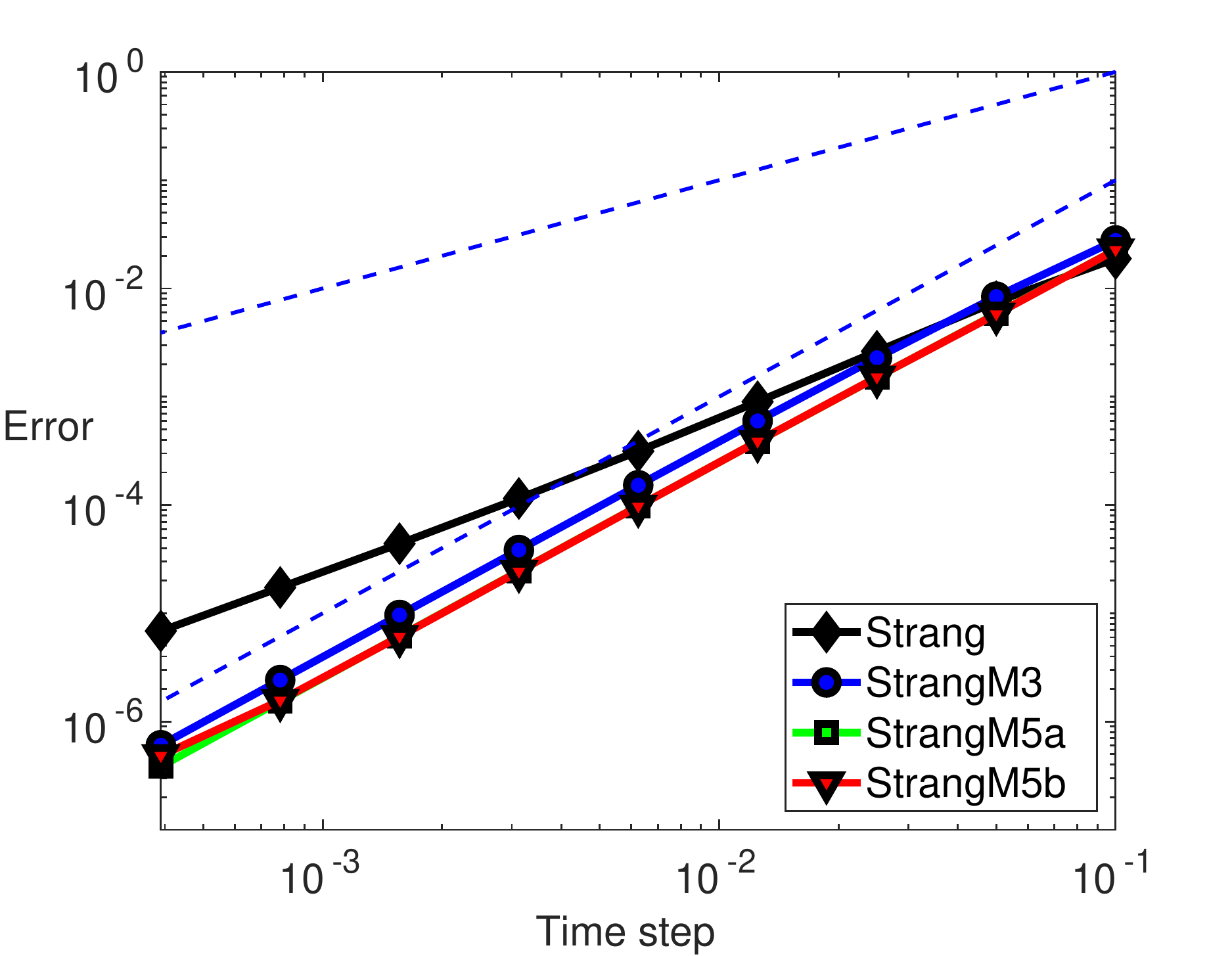} 
\caption*{Nonstiff case $M=1$.}
\end{subfigure}%
\begin{subfigure}{.5\textwidth}
\centering
\includegraphics[width=\textwidth]{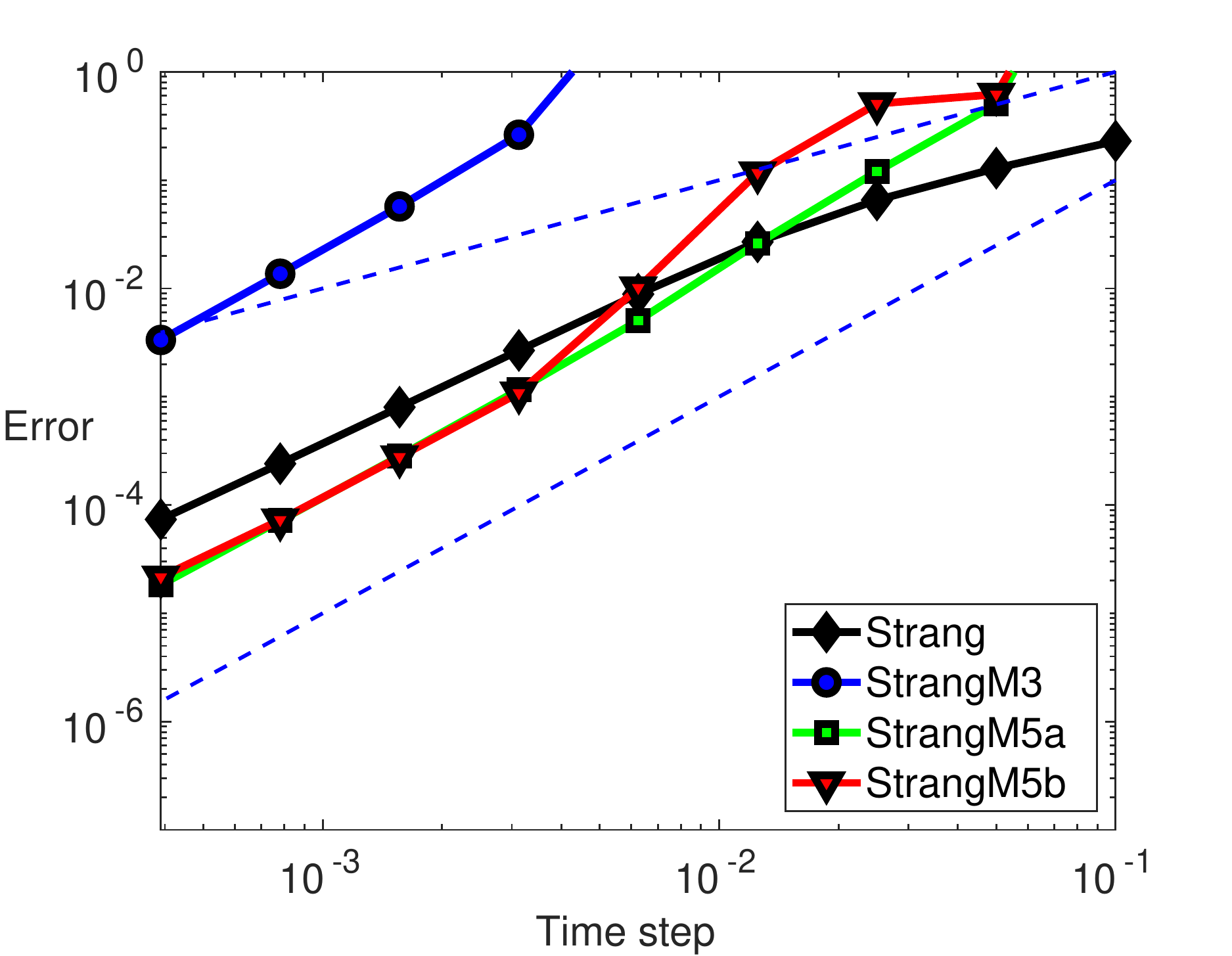}
\caption*{Stiff case $M=100$.}
\end{subfigure}
\begin{subfigure}{.5\textwidth}
\includegraphics[width=\textwidth]{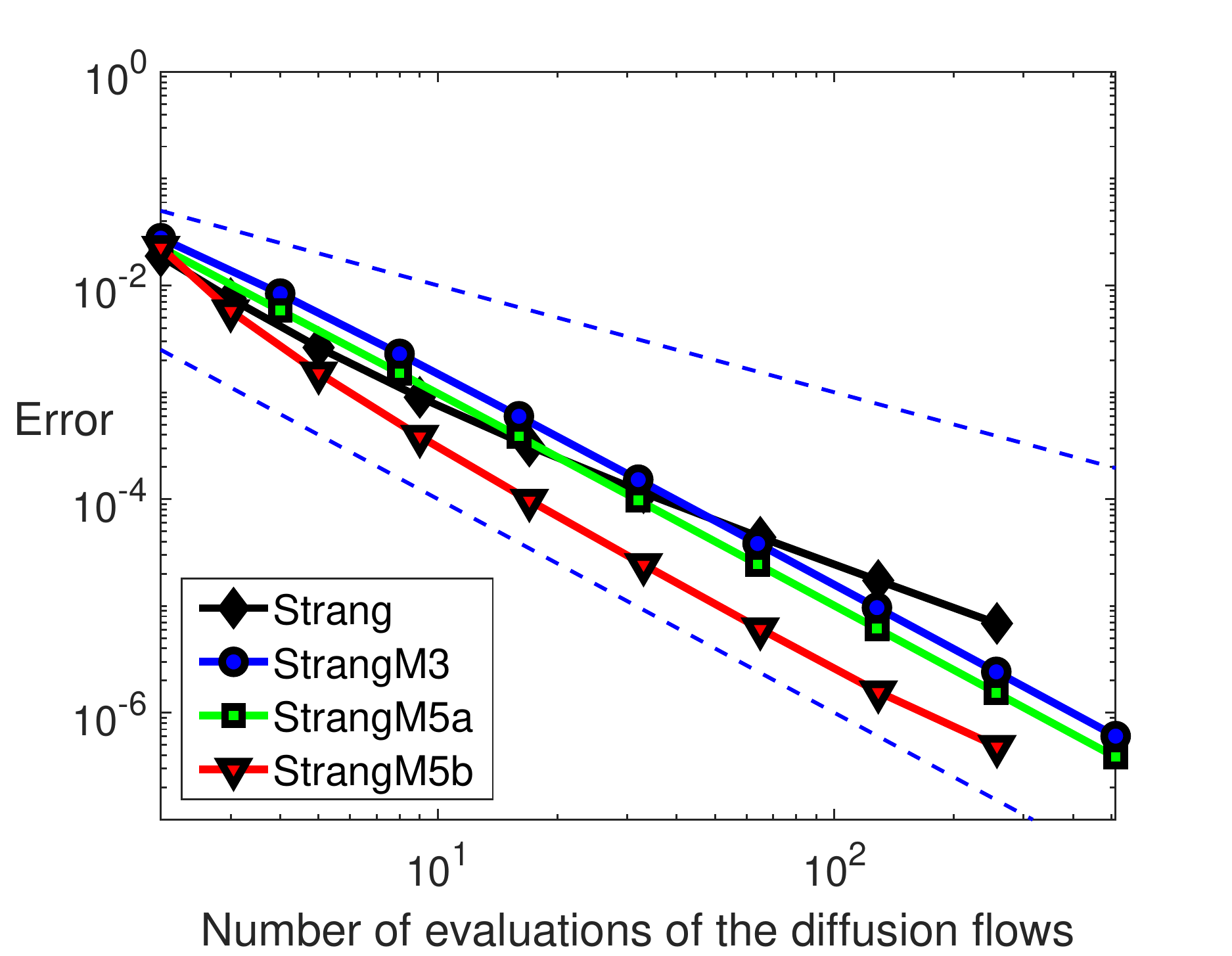} 
\caption*{Nonstiff case $M=1$.}
\end{subfigure}%
\begin{subfigure}{.5\textwidth}
\centering
\includegraphics[width=\textwidth]{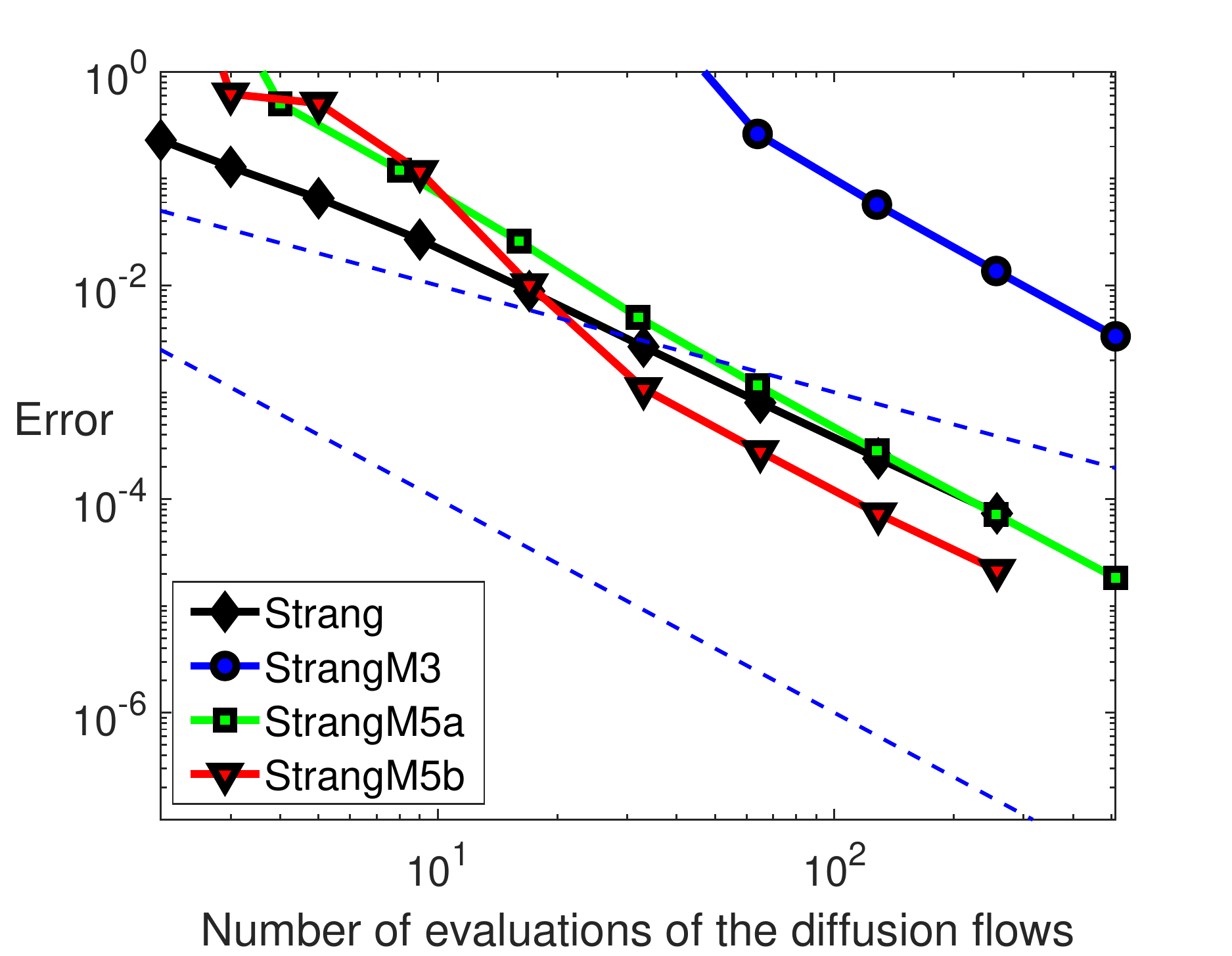}
\caption*{Stiff case $M=100$.}
\end{subfigure}
\caption{\textit{Comparison between the splitting methods \emph{Strang}~\eqref{eq : classical Strang 1}, \emph{StrangM3}~\eqref{eq : modified Strang 1}, \emph{StrangM5a}~(\ref{eq : five parts Strang splitting},~\ref{eq : q condition}) and \emph{StrangM5b}~(\ref{eq : five parts Strang splitting},~\ref{eq : q condition 2}) when applied to equation~\eqref{eq : stiff} with $M=1$ and $M=100$,  on $[0,1]^2$ with inhomogeneous mixed boundary conditions. Reference slopes one and two are given in dotted lines.}}\label{fig : Stiff}
\end{figure}
\\
We observe in the top pictures of Figure~\ref{fig : Stiff} that \emph{StrangM5a} and \emph{StrangM5b} are slightly more accurate than the modified Strang splitting \emph{StrangM3} for $M=1$ and drastically more accurate for the stiff case $M=100$. Indeed, as explained above, in this later stiff case, the accuracy of the splitting \emph{StrangM3} deteriorates due to the large correction involved. Moreover, in the bottom pictures of Figure~\ref{fig : Stiff}, we observe that for the same number of evaluations of the diffusion flows, \emph{StrangM5b} is more accurate than all the other methods also for $M=1$. Indeed, here, the diffusion flows dominates the total cost and we recall that \emph{StrangM5b} only requires one evaluation of the diffusion flow and one evaluation of the source term flow at each step (see Table~\ref{tab : Flows Evaluations} and Remark~\ref{rem : composition}), analogously to the classical Strang splitting \emph{Strang}.
\paragraph{Acknowledgements.} The authors would like to thank Christophe
Besse and Lukas Einkemmer for helpful discussions. This work was
partially supported by the Swiss National Science Foundation, grants No.
200021\_162404, No. 200020\_178752, No. 200020\_144313/1, No. 200020\_184614, No. 200021\_162404 and No. 200020\_178752.
\bibliographystyle{abbrv}
\bibliography{biblioUA}

\begin{thebibliography}{10}

\bibitem{Ein18}
L.~Einkemmer, M.~Moccaldi, and A.~Ostermann.
\newblock Efficient boundary corrected {S}trang splitting.
\newblock {\em Appl. Math. Comput.}, 332:76--89, 2018.

\bibitem{Ost15}
L.~Einkemmer and A.~Ostermann.
\newblock Overcoming order reduction in diffusion-reaction splitting. {P}art 1:
  {D}irichlet boundary conditions.
\newblock {\em SIAM J. Sci. Comput.}, 37(3):A1577--A1592, 2015.

\bibitem{Ost16}
L.~Einkemmer and A.~Ostermann.
\newblock Overcoming order reduction in diffusion-reaction splitting. {P}art 2:
  {O}blique boundary conditions.
\newblock {\em SIAM J. Sci. Comput.}, 38(6):A3741--A3757, 2016.

\bibitem{Fuj67}
D.~Fujiwara.
\newblock Concrete characterization of the domains of fractional powers of some
  elliptic differential operators of the second order.
\newblock {\em Proc. Japan Acad.}, 43:82--86, 1967.

\bibitem{Gri67}
P.~Grisvard.
\newblock Caract\'{e}risation de quelques espaces d'interpolation.
\newblock {\em Arch. Rational Mech. Anal.}, 25:40--63, 1967.

\bibitem{hlw10}
E.~Hairer, C.~Lubich, and G.~Wanner.
\newblock {\em Geometric numerical integration}, volume~31 of {\em Springer
  Series in Computational Mathematics}.
\newblock Springer, Heidelberg, 2010.
\newblock Structure-preserving algorithms for ordinary differential equations,
  Reprint of the second (2006) edition.

\bibitem{Jah00}
T.~Jahnke and C.~Lubich.
\newblock Error bounds for exponential operator splittings.
\newblock {\em BIT}, 40(4):735--744, 2000.

\bibitem{Lun95}
A.~Lunardi.
\newblock {\em Analytic semigroups and optimal regularity in parabolic
  problems}, volume~16 of {\em Progress in Nonlinear Differential Equations and
  their Applications}.
\newblock Birkh\"{a}user Verlag, Basel, 1995.

\bibitem{Nie12}
J.~Niesen and W.~M. Wright.
\newblock Algorithm 919: a {K}rylov subspace algorithm for evaluating the
  {$\phi$}-functions appearing in exponential integrators.
\newblock {\em ACM Trans. Math. Software}, 38(3):Art. 22, 19, 2012.

\bibitem{Vas92}
A.~S. Vasudeva~Murthy and J.~G. Verwer.
\newblock Solving parabolic integro-differential equations by an explicit
  integration method.
\newblock {\em J. Comput. Appl. Math.}, 39(1):121--132, 1992.

\end{thebibliography}
\end{document}